\documentclass[a4paper,3p,sort&compress]{elsarticle}
\usepackage[utf8]{inputenc}
\usepackage[english]{babel}
\usepackage{url}
\usepackage{amsmath,amssymb,amsfonts,amsthm}
\usepackage[font=scriptsize,labelfont=bf]{caption}
\usepackage{mathrsfs}
\usepackage{textcomp}
\usepackage{enumerate}
\usepackage{graphicx}
\usepackage{xcolor}
\usepackage[all,2cell]{xy}
\usepackage[ruled,linesnumbered, inoutnumbered]{algorithm2e}
\usepackage{tikz}
\usetikzlibrary{backgrounds,patterns,arrows,positioning,fit,shapes.geometric,decorations.pathmorphing}
\usepackage[capitalize]{cleveref}
\usepackage{blkarray}
\Crefname{equation}{}{}
\usepackage{comment}
\newtheorem{theorem}{Theorem}[section]
\newtheorem{lemma}[theorem]{Lemma}

\newtheorem{proposition}[theorem]{Proposition}
\newtheorem{corollary}[theorem]{Corollary}
\newtheorem{question}{Question}

\theoremstyle{definition}
\newtheorem{remark}[theorem]{Remark}
\newtheorem{definition}[theorem]{Definition}
\newtheorem{answer}{Answer}
\newtheorem{example}{Example}
\newcommand{\changed}[1]{#1}

\DeclareMathAlphabet{\mathpzc}{OT1}{pzc}{m}{it}

\newcommand{\rank}{\mathrm{rk}\,}

\newcommand{\Var}[1]{\mathcal{#1}}

\newcommand{\Tang}[2]{T_{#1} {#2}}

\newcommand{\R}{\mathbb{R}}
\newcommand{\deriv}[2]{\mathrm{d}_{#2}#1}

\newcommand{\dist}{\mathrm{dist}}

\newcommand{\bbk}{\Bbbk}

\newcommand{\bbC}{\mathbb{C}}
\newcommand{\bbN}{\mathbb{N}}
\newcommand{\calM}{\mathcal{M}}
\newcommand{\calP}{\mathcal{P}}
\newcommand{\calX}{\mathcal{X}}
\newcommand{\calY}{\mathcal{Y}}

\newcommand{\bbR}{\mathbb{R}}

\newcommand{\vvirg}{, \dots , }

\renewcommand{\bar}[1]{\overline{#1}}
\newcommand{\im}{\mathrm{im}}

\newcommand{\textprod}{{\textstyle \prod}}
\newcommand{\bfx}{\mathbf{x}}

\crefname{equation}{}{}
\crefname{equation}{}{}
\crefname{figure}{Figure}{Figures}
\crefname{section}{Section}{Sections}
\crefname{lemma}{Lemma}{Lemmata}
\crefname{prop}{Proposition}{Propositions}
\crefname{thm}{Theorem}{Theorems}
\crefname{cor}{Corollary}{Corollaries}
\crefname{question}{Question}{Questions}
\crefname{dfn}{Definition}{Definitions}
\crefname{notation}{Notations}{Notations}
\crefname{rem}{Remark}{Remarks}
\crefname{claim}{Claim}{claims}
\crefname{assumption}{Assumption}{Assumptions}
\crefname{proposition}{Proposition}{Proposition}
\crefname{corollary}{Corollary}{Corollaries}
\def\CC{\mathbb{C}}
\def\RR{\mathbb{R}}

\allowdisplaybreaks

\begin{document}

\author{Paul Breiding\fnref{fund1}}
\fntext[fund1]{Supported by the Deutsche Forschungsgemeinschaft (DFG) -- Projektnummer 445466444.}
\ead{pbreiding@uni-osnabrueck.de}
\address{University Osnabr\"uck,
Fachbereich Mathematik/Informatik
Albrechtstr.\ 28a,
49076 Osnabrück, Germany}
\author{Fulvio Gesmundo}
\ead{gesmundo@cs.uni-saarland.de}
\address{Saarland University, Saarland Informatics Campus, 66123 Saarbr\"ucken, Germany}
\author{Mateusz Micha{\l}ek\fnref{fund2}}
\fntext[fund2]{Supported by the Deutsche Forschungsgemeinschaft (DFG) -- Projektnummer 467575307.}
\ead{mateusz.michalek@uni-konstanz.de}
\address{University of Konstanz, Dept.~of Mathematics and Statistics, Universit\"atsstrasse 10, 78457 Konstanz, Germany}
\author{Nick Vannieuwenhoven\fnref{fund3}}
\ead{nick.vannieuwenhoven@kuleuven.be}
\address{KU Leuven, Department of Computer Science, Celestijnenlaan 200A, B-3001 Leuven, Belgium;\\ Leuven.AI, KU Leuven Institute for AI, B-3000 Leuven, Belgium}
\fntext[fund3]{Partially supported by a Postdoctoral Fellowship of the Research Foundation---Flanders (FWO) with project 12E8119N. Partially supported by Internal Funds KU Leuven BOF STG/19/002.}

\title{Algebraic compressed sensing}

\begin{abstract}
We introduce the broad subclass of algebraic compressed sensing problems, where structured signals are modeled either explicitly or implicitly via polynomials. This includes, for instance, low-rank matrix and tensor recovery. We employ powerful techniques from algebraic geometry to study well-posedness of sufficiently general compressed sensing problems, including existence, local recoverability, global uniqueness, and local smoothness. Our main results are summarized in thirteen questions and answers in algebraic compressed sensing. Most of our answers concerning the minimum number of required measurements for existence, recoverability, and uniqueness of algebraic compressed sensing problems are optimal and depend only on the dimension of the model.

\end{abstract}

\maketitle

\section{Introduction} \label{sec_introduction}

Compressed sensing investigates inverse problems in which a structured signal in a Euclidean space is to be reconstructed from a smaller number of (typically linear) measurements. Such inverse problems are well posed in the sense of Hadamard if there exists a unique, continuous solution.

The general setting consists of a \emph{model} $\Var{X}\subset \mathbb R^n$ and a \emph{measurement map}
\begin{align}\tag{IP}\label{eqn_ip}
\mu: \Var{X}\to \mathbb R^s.
\end{align}
Given a \emph{measurement} $y \in \mu(\calX)$, one is interested in reconstructing a \textit{signal} $x \in \Var{X}$ such that $y = \mu(x)$. This setting is schematically depicted in Figure \ref{fig1}.

We mainly focus on linear measurements, but we also consider the case where~$\mu$ is nonlinear. Recall that $\mu$ is a \textit{linear measurement map} if
\[
\mu = \hat{\mu}|_\Var{X}, \text{ where } \hat{\mu}:\bbR^n\to\bbR^s \text{ is a linear map}.
\]
We say that~$s$ is the \emph{number of measurements}. The \textit{forward map}~$\mu$ that takes a signal $x\in \Var{X}$ to the observed measurement $y=\mu(x)$ is usually easy to compute (and sometimes is even implemented in hardware \cite{DE2011}). On the other hand, the inverse problem that consists of computing $x\in \mu^{-1}(y)$ on input $y\in\mu(\Var{X})$ is generally a more difficult problem, and in fact it is not always possible to solve.

According to Kirsch \cite{Kirsch2011}, an inverse problem is \textit{well posed} in the sense of Hadamard at $y \in \bbR^s$ if the following conditions simultaneously hold:
\begin{enumerate}
 \item[(E)] \textit{Existence:} there exists a solution, i.e., there exists $x \in \calX$, with $y = \mu(x)$;
 \item[(I)] \textit{Identifiability:} the solution is locally unique at $y$, that is, there is an open neighborhood $U$ of $y$ in $\mu(\calX)$ such that, for every $y' \in U$, there is a unique $x' \in \calX$ with $y' = \mu(x')$;
 \item[(C)] \textit{Continuity:} the solution $x' = \mu^{-1}(y')$ is continuous as a function of $y' \in U$.
\end{enumerate}
It is argued in \cite[Chapter 1]{Kirsch2011} that well-posedness is a practical necessity for numerically solving the inverse problem associated with \cref{eqn_ip}.

\begin{figure}[t]
\begin{center}
\begin{tikzpicture}[remember picture,scale=1]
\node[inner sep=0] at (0,0)
{\includegraphics[width=.75\textwidth]{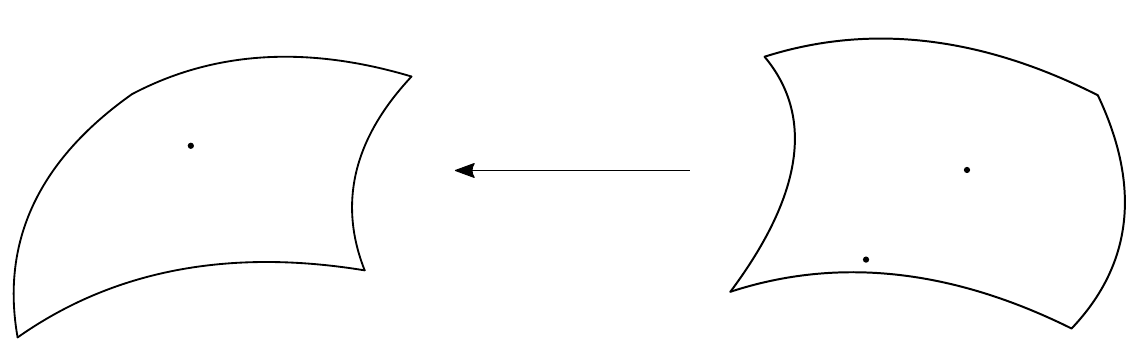}};
\node at (4.0,0.2) {$x_1$};
\node at (2.9,-0.8) {$x_2$};
\node at (-3.8,0.2) {$y$};
\node at (0,0.3) {$\mu$};
\node at (2.9,-2) {model $\Var{X}\subset \mathbb R^n$};
\node at (-3.4,-2) {measurements $\mu(\Var{X})\subset \mathbb R^s$};
\end{tikzpicture}
\caption{Schematic description of the geometry. On the right is the model $\Var{X}$. On the left is the set of measurements $\mu(\Var{X})$ taken from~$\Var{X}$ using the measurement map $\mu$. In the picture we have $\mu(x_1) = y$. We are concerned with the question whether $x_1$ is the only point in $\Var{X}$ mapping to $y$, or whether there exists another point $x_2\in \Var{X}$ with $\mu(x_2)=y$.}\label{fig1}
\end{center}
\end{figure}
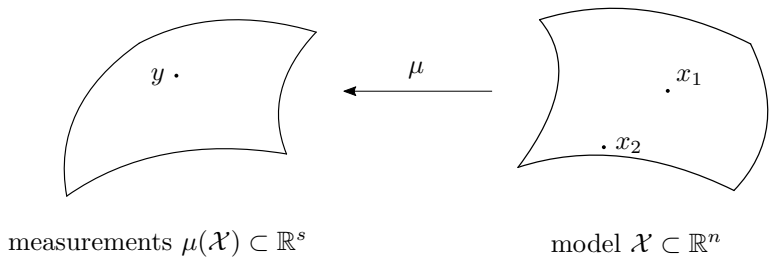

Special cases of the inverse problem \cref{eqn_ip} were extensively studied in compressed sensing \cite{Donoho2006,CT2005,CT2006,CRT2006b,CRT2006}. In the original formulation, $\Var{X} \subset \R^n$ is a (special) union of linear subspaces and $\mu$ is the restriction to $\calX$ of a linear map $\hat{\mu}:\mathbb R^n \to \mathbb R^s, x\mapsto Ax$ \cite{DE2011,BW2009}. Later, more general models for $\Var{X}$ were investigated, such as smooth manifolds \cite{BW2009,BCW2010}. See \Cref{ex_lowrankmatrices,ex_tensorrank,ex_moment_varieties} below for further examples of compressed sensing.

Many techniques were developed in compressed sensing to prove well-posedness of a problem, such as mutual coherence, the restricted isometry property (RIP), the nullspace property, or bi-Lipschitz continuity \cite{FR2013}. Arguably, the most popular of these is the RIP \cite{Donoho2006,CT2005,CT2006}. However, as we discuss in \cref{sec:review}, it seems that RIP-based analyses require a number of measurements $s \ge O( d \log n)$ to guarantee well-posedness, where~$d$ is the intrinsic dimension of the model $\Var{X}$ and $n$ is the dimension of the ambient space. One main contribution of this paper is reducing the number of required samples for (\textit{almost everywhere}) well-posedness to the minimal value for an important class of models $\Var{X}$ and measurement maps $\mu$ that includes the original union of subspaces and certain smooth manifolds.

This paper studies the well-posedness of \textit{algebraic compressed sensing}; that is, inverse problems where the model $\calX$ is explicitly or implicitly defined by polynomials and $\mu$ is a polynomial map. Our main tools rely on fundamental techniques from classic algebraic geometry.

We treat the following two complementary settings simultaneously in this work:
\begin{itemize}
 \item[(IS)] \label{def-X1} \textit{Implicit setting}: $\calX$ is defined as the \textit{real vanishing locus} of a real polynomial system of equations. More precisely,
 \begin{equation*}
\Var{X} = \{ x\in\R^n : f_1(x)=\dots=f_k(x)=0 \},
\end{equation*}
where $f_1,\dots,f_k\in \bbR[x_1,\dots,x_n]$ are polynomials with real coefficients in $n$ real variables.
\item[(ES)] \label{def-X2} \textit{Explicit setting}: $\calX$ is defined as the \emph{image of a polynomial map}, so that
\begin{equation*}
\Var{X} = \{ x=\phi(t)=(\phi_1(t), \dots, \phi_n(t)) : t \in \mathbb R^\ell \},
\end{equation*}
where $\phi_1,\dots,\phi_n \in\bbR[t_1,\dots,t_\ell]$ are polynomials with real coefficients in $\ell$ real variables.
\end{itemize}
Most of our results hold for complex models as well, so for completeness we also cover this case.
For convenience and clarity of expression, we assume that $\Var{X}$ is \emph{irreducible}: this technical condition is explained in \cref{sec:AG}. Intuitively, it means that $\Var{X}$ cannot be expressed as a finite union of such sets. As the \emph{affine space} $\mathbb R^\ell$ is irreducible and images of irreducible sets are irreducible, irreducibility always holds in the setting (ES). With some carefulness about the interpretation of genericity, most of our results can be extended to reducible varieties in setting (IS) as well.

In setting (IS) we say that $\Var{X}$ is \textit{implicitly defined} by polynomials, and in setting (ES) we say that $\Var{X}$ is \emph{explicitly parametrized} by polynomials.
In the language of algebraic geometry, a set implicitly defined by polynomials is called a \emph{real algebraic variety}, while sets explicitly parameterized by polynomials are \emph{semi-algebraic sets}. Though note that not all semi-algebraic sets arise in this way.

These two classes of polynomial models cover a wide range of specific cases, many of which were already studied individually in a compressed sensing context.
The following examples demonstrate that our assumptions on $\Var{X}$ are not very restrictive and the results of this work can be applied in a number of different settings. For an overview of polynomial models in the sciences see also \cite{breiding2021nonlinear}.

\begin{example}[Classic compressed sensing]
The original formulation of compressed sensing assumed that an observed signal $y \in \bbR^s$ is a $k$-sparse linear combination of a set of~$m$ vectors $v_1,\dots,v_m$ in $\bbR^n$. That is, there exist indices $i_j$, $j=1,\ldots,k$, so that
$
 y = \hat{\mu}( \lambda_1 v_{i_1} + \lambda_2 v_{i_2} + \dots + \lambda_k v_{i_k} ),
$
where $\hat{\mu}:\bbR^n\to\bbR^s$ is a linear map.
This is an algebraic compressed sensing problem where $\Var{X}\subset\bbR^n$ is the union of the $\binom{m}{k}$ linear subspaces of dimension $k$ spanned by the subsets of $k$ vectors from $v_1,\dots,v_m$. Such a union of subspaces can be formulated as set of solutions of a polynomial system in setting (IS). The sensing map is $\mu = \hat{\mu}|_\Var{X}$.

Note that $\Var{X}$ is not irreducible. Nevertheless, the well-posedness of such compressed sensing problems is well known \cite{FR2013} and it can be verified that for a union of equidimensional varieties, the main result, \cref{demo-thm}, applies verbatim.
\end{example}

\begin{example}[Matrix completion]\label{ex_lowrankmatrices}
The algebraic variety of bounded-rank matrices in~$\bbR^{m\times n}$ is a model for compressed sensing that was studied in \cite{FCRP2008}. When the measurement map is a coordinate projection, the associated inverse problem is also called a \textit{matrix completion problem}, which has many applications, including data analysis and recommender systems \cite{NKS2019}. Polynomials that implicitly define $m\times n$ matrices of rank at most $r$ are all the $(r+1)\times (r+1)$ minors. One can also view low-rank matrices in the explicit setting, via the polynomial parameterization given by $\phi:\bbR^{m\times r} \times \bbR^{n\times r} \to \bbR^{m\times n}, (A,B)\mapsto AB^T$. Related low-rank matrix models that our analysis covers include low-rank matrices with sparseness constraints \cite{GV2012} and low-rank plus sparse matrix recovery \cite{CSPW2009,TV2020}. The case of measurement maps that are coordinate projections is addressed in \cite{BR2000,SC2010,KTT2015,GHIL:complexity_lin_circuits_and_geometry,Tsakiris2021,Tsakiris2021b} \changed{using a variety of techniques from rigidity theory, algebraic combinatorics, and matroid theory.} 
\end{example}

\begin{example}[Tensor completion]\label{ex_tensorrank}
Other examples are various bounded-rank tensor decompositions \cite{KB2009,GKT2013}, such as the tensor rank decomposition or canonical polyadic decomposition \cite{Hitchcock1}, Tucker decomposition \cite{Tucker1966}, tensor trains \cite{Vidal2003,Oseledets2011}, hierarchical Tucker \cite{HK2009,Grasedyck2010}, and even general tensor network decompositions \cite{YL2018}. All of these examples can be parameterized straightforwardly by polynomials. These decompositions feature in data analysis applications as an a priori model for filling in missing data \cite{KB2009,PFS2017}.
The compressed sensing literature already investigated well-posedness conditions for the tensor rank decomposition with sparsity constraints \cite{SK2012}, unconstrained tensor rank decomposition \cite{RS2020}, and Tucker decomposition, tensor trains, and hierarchical Tucker decompositions were studied jointly in \cite{RSS2017}.
\end{example}

\begin{example}[Moment varieties]\label{ex_moment_varieties}
Our final example is from (algebraic) statistics \cite{algstat2018}. Multivariate probability distributions are characterized by an infinite sequence of symmetric tensors of increasing degree, namely their \textit{sequence of moments}, or, alternatively, cumulants \cite{Kolassa2006}. Many probability distributions belong to a parameterized family; for example, $k$-variate Gaussian distributions are parameterized by the mean vector $m\in\bbR^k$ and the covariance matrix $\Sigma \in \bbR^{k\times k}$. Pearson \cite{Pearson1894} observed that for some families a finite sequence of moments suffices to identify the parameters of the distribution from empirical moments. In recent years, it was realized that the set of possible $d$th order moments of a parameterized family is sometimes defined by polynomials \cite{AAR2021,AFS2016,KSS2020},
thus carving out a \textit{moment variety} $\Var{X}$ in the space of $d$th order symmetric tensors in $k$ variables $S^d \bbR^k \simeq \bbR^{\binom{k+d-1}{d}}$. Estimating \textit{all} $\binom{k+d-1}{d}$ entries of the $d$th order empirical moment in $S^d \bbR^k$ is computationally challenging. However, since the probability distribution is represented by a point on $\Var{X} \subset S^d \bbR^k$, we could instead compute a linear projection of the empirical moment. This leads to an algebraic compressed sensing problem. Only recently, the well-posedness of this problem was studied by Lindberg, Amendola, and Rodriguez \cite{LAR2021}.
For mixtures of Gaussians, Guo, Nie and Yang \cite{GNY2021} recently presented an efficient numerical algorithm to estimate the parameters from coordinate projections of the third-order moment.
\end{example}

Let
\[
d:=\dim \Var{X}
\]
denote the dimension of the model $\calX$. The notion of dimension is introduced formally in Section \ref{subsec: dimension and tangent}. This is the formal way to measure the ``intrinsic size'' of the model $\Var{X}$; it coincides with the dimension of the set of non-singular points of $\Var{X}$, which is a smooth manifold embedded in $\bbR^n$ \cite{BR1990}. In the explicit setting (ES), $\dim \calX$ is the maximal rank of the Jacobian of the parametrization $\phi$ at points in~$\mathbb R^\ell$ \cite{CLO2015}; in particular, in this setting, $\dim \calX$ is bounded from above by $\ell$, which is often called the number of \emph{degrees of freedom} in the parameterization. \changed{The latter can be strictly larger than the dimension; the explicit parameterization in \cref{ex_lowrankmatrices}, for example, has $(m+n)r$ degrees of freedom, while the dimension of the variety of bounded-rank matrices is only $(m+n-r)r$ \cite[Proposition 12.2]{Harris1992}.} As we will see, the dimension of the model~$\Var{X}$ determines the required number of measurements $s$ so that the inverse problem is well posed. Throughout this paper, the dimension $d$ is fixed and the number of measurements $s$ can vary, reflecting the design choice of the measurement map.

The heart of this paper is \cref{q_and_a}. We discuss and explore \textbf{thirteen questions and answers} about the well-posedness of algebraic compressed sensing problems. The questions are divided into four themes:
\begin{enumerate}
\item[(E)] \emph{Existence:} Given $y\in\R^s$ when are there signals $x\in \Var{X}$ with $y = \mu(x)$?
\item[(R)] \emph{Recoverability}: When are there finitely many points in $\mu^{-1}(y)$ for $y\in\mu(\Var{X})$?
\item[(I)] \emph{Identifiability}: When is $\mu^{-1}(\mu(x))$ equal to $\{x\}$?
\item[(C)] \emph{Continuity:} When is $\mu^{-1}$ locally (Lipschitz) continuous at $\mu(x)$?
\end{enumerate}
If (E), (I) and (C) hold at $y$, the inverse problem is well posed at $y$ by the definition we gave above. The main reason to consider recoverability as well is that it enables localized well-posed algebraic compressed sensing problems. Indeed, if only (E), (R) and (C) hold at $y$, then the inverse problem can be made well posed by further restricting the variety $\Var{X}$ to a neighborhood of one of the elements in the fiber of $\mu$.

The engine that powers our answers in the Q\&A section is a fundamental result from commutative algebra called the \emph{Noether Normalization Lemma}; see \cref{thm:Noether Norm} below. The proof of this result is sketched in \cref{sec:AG} and we refer to \cite{Eis:CommutativeAlgebra} for a comprehensive discussion on the proof and the lemma's deep implications.

Many of our answers use the notion of a ``generic point in $\Var{X}$''. This is a subtle concept, which we carefully explain in \cref{sec:AG}. For the moment, a reader with an affinity to probability may think of a generic point as a point  being chosen with probability 1 from a probability measure that is absolutely continuous with respect to the Lebesgue measure; a reader with a background in analysis can interpret the set of generic points as an open dense subset in $\Var{X}$. Nevertheless, being a generic point is a strictly stronger property.

As a teaser of the questions and answers, we summarize in the next theorem the answers to \cref{QExistence,question:finite map,Q3,Q11}. In \cref{q_and_a}, we list nine additional questions and answers, which shed light on different aspects of algebraic compressed sensing.
\begin{theorem}\label{demo-thm}
Let $\Var{X}$ be given as in (IS) or (ES), and let $\mu$ be a generic linear map $\mu: \Var{X}\rightarrow \mathbb R^s$. Then:
\begin{enumerate}
\item[(E)] the fiber $\mu^{-1}(y) \ne \emptyset$ at $y\in\R^s$ if and only if a particular system of polynomial equations has a solution:
\begin{itemize}
\item in the implicit setting (IS), this system of equations in the variables $x=(x_1,\ldots,x_n)$ is defined by
$f_1(x)=\cdots=f_k(x)=\mu(x)-y= 0$;
\item in the explicit setting (ES), this system in variables $t=(t_1,\ldots,t_\ell)$ is defined by
$\mu(\phi(t))-y=0;$
\end{itemize}
\item[(R)] if $s\ge d$, the fiber $\mu^{-1}(\mu(x))$ of a generic point $x\in \Var{X}$ is finite;
\item[(I)] if $s\geq d+1$, the fiber $\mu^{-1}(\mu(x))$ of a generic point $x\in \Var{X}$ is equal to $\{x\}$;
\item[(C)] if $s\geq d+1$, the fiber $\mu^{-1}(\mu(x))$ of a generic point $x\in \Var{X}$ is $C^\infty$ continuous.
\end{enumerate}
\end{theorem}

The paper is structured as follows. In the next section, we offer a glimpse of what lies beyond the restricted isometry property for establishing (local) well-posedness of inverse problems like \cref{eqn_ip}. Thereafter, in \cref{q_and_a}, the main results of this paper are presented as a sequence of thirteen basic questions and answers related to the well-posedness of algebraic compressed sensing problems. \Cref{sec:AG} recalls the main techniques from algebraic geometry that are employed in \cref{q_and_a_proofs} to furnish proofs of the main results from \cref{q_and_a}. Numerical experiments in \cref{sec_numexp} illustrate the minimal requirements on the number of measurements for well-posedness in algebraic compressed sensing. \changed{The paper is concluded in \cref{sec_conclusions} with an outlook and open problems in algebraic compressed sensing.}

\section{What lies beyond the restricted isometry property?}\label{sec:review}

The $\epsilon$-\textit{restricted isometry property} ($\epsilon$-RIP) of the measurement map $\mu : \Var{X} \to \bbk^s$ \cite{CT2005,CT2006,Donoho2006}, where $\bbk=\bbR$ or $\bbk=\bbC$, is a core concept in compressed sensing. For general models $\Var{X} \subset \bbk^n$, it can be defined as follows \cite{BCW2010}: for all $x, y \in \Var{X}$ we have
\begin{align}\tag{RIP}\label{eqn_rip}
 0 < 1-\epsilon \le \frac{\| \mu(x) - \mu(y) \|_{\bbk^s}}{\| x - y \|_{\bbk^{n}}} \le 1+\epsilon < \infty,
\end{align}
where $\epsilon \in [0,1)$ is a constant depending on the measurement map $\mu$, and $\| \cdot \|_A$  denotes a norm on the vector space $A$.
The $\epsilon$-RIP implies that the measurement map~$\mu$ is everywhere identifiable.
Maps satisfying the RIP are a special type of \textit{global bi-Lipschitz embeddings}. Such embeddings were studied earlier for general metric spaces in \cite{Assouad1983,JL1984,LV1977,MovahediLankerani1990}.
Recall that a map $\mu : \Var{X} \to \Var{Y}$ between metric spaces $(\Var{X},\dist_\Var{X})$ and $(\Var{Y}, \dist_\Var{Y})$ is \textit{bi-Lipschitz} if there exist constants $c, C$ such that
\[
 0 < c \le \frac{\dist_\Var{Y}( \mu(x),  \mu(y) )}{\dist_\Var{X}( x, y ) } \le C < \infty
\]
for all $x \ne y$ \cite{BBI2001}.
Geometrically this means that $\mu$ is an injective map into $\bbk^s$ that does not stretch or bend $\Var{X}$ too much.
For this reason, bi-Lipschitz embeddings are sometimes called \textit{stable embeddings} \cite{BW2009,BCW2010}.
A bi-Lipschitz map has a bi-Lipschitz inverse map $\mu^{-1}$ with lower and upper Lipschitz constants equal to respectively~$C^{-1}$ and $c^{-1}$. Bi-Lipschitz embeddings are fundamental in compressed sensing because they offer two complementary \emph{global} properties:
\begin{enumerate}
 \item[(I)] the map is everywhere injective, and
 \item[(II)] the inverse map is Lipschitz continuous.
\end{enumerate}
The first property ensures that for all measurements $y \in \mu(\calX)$, the inverse problem has a unique solution. The second property guarantees that the inverse problem has a \textit{bounded sensitivity}, i.e., the inverse map has a bounded variation from one point to another, relative to the distance between the points.
Indeed, both the $\epsilon$-RIP and the $(1-\epsilon,C)$-bi-Lipschitz condition offer the \textit{uniform} upper bound
\begin{align} \tag{Lip}  \label{eqn_err}
 \dist_\Var{X}\left( \mu^{-1}(y_1), \mu^{-1}(y_2) \right) \le \frac{\dist_\Var{Y}( y_1, y_2 )}{1 - \epsilon}
\end{align}
for all measurements $y_1, y_2 \in \mu(\Var{X})$.
Taken together, (I) and (II) thus guarantee that the inverse problem \cref{eqn_ip} is well posed.

While properties (I) and (II) are certainly desirable, demanding that the measurement map is \textit{globally} bi-Lipschitz can be challenging to verify theoretically and imposes severe geometric constraints on $\Var{X}$ and $\mu$. Moreover, computing the optimal RIP constant for classic compressed sensing of $k$-sparse vectors is an NP-hard problem \cite{TP2014}. Our viewpoint is that these constraints can sometimes be relaxed. We believe that the more flexible, localized, easily implementable analysis as the one presented in this work can better cater to practical situations.

\subsection{How many measurements do we need?}
As the intrinsic dimension of the inputs is only $d = \dim \Var{X}$, we could a priori hope that $s = d$ measurements would suffice to recover $x$ from $\mu(x)$. However, $\epsilon$-RIP-based analyses of stable recovery of linear maps $\mu : \Var{X} \to \bbk^s$ in the literature ultimately lead to a required number of samples $s$ that scales at least like
\(
 d \cdot \log \left( \frac{n}{\epsilon^2} \right)
\)
when $\Var{X} \subset \bbk^n$ is a collection of $d$-dimensional linear spaces; see, e.g., \cite{BW2009,DE2011}. For compact $d$-dimensional manifolds a similar requirement is needed \cite[Theorem~3.1]{BW2009}.

The ratio between the intrinsic dimension $d$ and the required number of measurements $s$ for recoverability is called the \textit{oversampling rate}. So this rate seems to scale logarithmically in $n \epsilon^{-2}$ in the case where $\calX$ is a union of linear spaces. When considering more general models $\Var{X}$, insisting on a bi-Lipschitz embedding could lead to even higher oversampling rates. For general Lipschitz manifolds $\Var{X}$, the best upper bound we are aware of on the minimal number of measurements $s$ to recover any model $\Var{X}$ with (only) a \textit{local}\footnote{A map $f$ is a local embedding if every point in the domain has a neighborhood such that $f$ restricted to that neighborhood is an embedding.} bi-Lipschitz embedding $\mu$ is $s \ge d(d+1)$ \cite{LV1977}. Such a quadratic growth is often too high in applications.

As bi-Lipschitz embeddings preserve much more information about the model~$\Var{X}$ than mere injective maps, it is not surprising that RIP or bi-Lipschitz conditions lead to an oversampling factor greater than $1$.
But is this really \textit{necessary} for recoverability?
By relaxing the requirement of bi-Lipschitz embeddability, recoverability and identifiability can often be guaranteed with a smaller number of measurements.

\textit{This observation is neither new nor controversial.} In 1936, Whitney~\cite{Whitney1936} proved that a smooth $d$-dimensional manifold $\Var{X}$ can be smoothly embedded (in particular, the map is injective and its inverse is smooth) into a Euclidean space of dimension $s = 2 d$ for some smooth $\mu$. The oversampling rate is thus $2$. With such a measurement map, every point is identifiable and $\mu^{-1}$ is a local Lipschitz map, but $\mu^{-1}$ is not necessarily a (global) Lipschitz map.
Birbrair, Fernandes, and Jelonek \cite{BFJ2021} recently proved that if $\Var{X}$ is a closed semi-algebraic set then there exist bi-Lipschitz linear maps $\mu$ into $\R^{2 d + 1}$. Hence, even with bi-Lipschitz embeddings one can get to an oversampling factor as low as $2 + o(1)$ for some models~$\Var{X}$. Yet, we show that for a generic linear map $d$ measurements suffice for generic recoverability and $d+1$ measurements suffice for generic identifiabilty; see \cref{demo-thm}(R) and (I).

\subsection{How much does the reconstruction move?}
Besides identifiability, the other main use of the RIP is deriving upper bounds on the reconstruction error of specific algorithms \cite[Chapter 6]{FR2013}. In most of such analyses, a reconstruction algorithm is called \textit{stable} or \textit{robust} if, in some norm, the reconstruction error $\| x - x^\sharp \|$ is small. Here, $x^\sharp$ is the algorithm's output on input $y=\mu(x) + e$ and $e$ is a small error. For statements about the main algorithms in traditional compressed sensing see also \cite[Theorems 6.12, 6.21, 6.25, 6.28]{FR2013}.

The concept of \textit{stability} has many meanings in numerical analysis. \textit{Robustness} of a reconstruction algorithm in compressed sensing is different from but closely related to the \textit{forward stability} of an algorithm in numerical linear algebra.
Forward stability of an algorithm $\hat{f}$ for a function $f$ is defined in terms of \textit{condition numbers} \cite{TB1997}.
Following \cite{Rice1966}, the condition number~$\kappa[f](x)$ of a map $f : \Var{X} \to \Var{Y}$ between metric spaces $(\Var{X},\dist_\Var{X})$ and $(\Var{Y},\dist_\Var{Y})$ at a point~$x\in \Var{X}$ is
\begin{align} \tag{CN} \label{eqn_kappa}
\kappa[f](x) = \lim_{\epsilon\to 0} \sup_{\substack{y \in \Var{X},\\ \dist_\Var{X}(x,y) \le \epsilon}} \frac{\dist_\Var{Y}(f(x), f(y))}{\dist_\Var{X}(x,y)}.
\end{align}
The condition number is an intrinsic property of the function $f$. It measures how much an infinitesimal perturbation of the input is amplified in the output of the function $f$. It is a standard way to measure the \textit{sensitivity} to perturbations of a computational problem modeled by the function~$f$ \cite{TB1997,BC2013,Rice1966}. If $f$ is a differentiable function, the condition number provides an error bound of the form
\[
 \dist_\Var{Y}( f(x), f(y) ) \le \kappa[f](x) \cdot \dist_\Var{X}(x,y) + o(\dist_\Var{X}(x,y))
\]
and this bound is asymptotically sharp as $\dist_\Var{X}(x,y) \to 0$. Aside from the foregoing error bound, the condition number is also viewed as a measure of complexity or hardness of computing the output of $f$ \cite{BCSS,BC2013}.

In our setting, the role of $f$ is played by the measurement map $\mu$. An algorithm $\hat{\mu}^{-1}$ for the function $\mu^{-1}$ is called \textit{forward stable} \cite{TB1997,Higham1996} if its reconstruction error is ``not much larger'' than the \textit{backward error} (the difference between~$\mu(x)$ and~$\mu(x^\sharp)$) multiplied by the condition number.
More precisely, a compressed sensing reconstruction algorithm is forward stable if
\begin{align}\label{eqn_stab}\tag{FS}
 \underbrace{\| x - x^\sharp \|_{\bbk^n}}_\text{\changed{output error}} \le \underbrace{K}_\text{algorithm-specific} \cdot \underbrace{\kappa[\mu^{-1}](\mu(x^\sharp))}_\text{problem-specific} \cdot \underbrace{\| \mu(x) - \mu(x^\sharp) \|_{\bbk^s}}_\text{\changed{backward error}},
\end{align}
where \changed{$1 \le K < \infty$} is a constant, $\kappa[\mu^{-1}](\mu(x^\sharp))$ is the condition number relative to the distances $\dist_\Var{X}(x,x')=\|x-x'\|_{\bbk^n}$ and $\dist_{\bbk^s}(y,y')=\|y-y'\|_{\bbk^s}$, and the backward error is $\|\mu(x) - \mu(x^\sharp)\|$.
As indicated above, the condition number is a mathematical property of $\mu^{-1}$. The constant $K$ quantifies how much the algorithm $\hat{\mu}^{-1}$ amplifies the inherent sensitivity of the problem $\mu^{-1}$. Note that the algorithm-specific part of the error must be constant, while the problem-specific part depends on the problem instance that is solved.
An algorithm is forward stable if $K$ is a constant that is ``small'' in the eye of the beholder.

Inspecting \cite[Theorems 6.12, 6.21, 6.25, 6.28]{FR2013}, for example, it is not difficult to prove that all these bounds imply that the corresponding algorithms are forward stable, with an undetermined constant $K$ in \cref{eqn_stab}. The reason is as follows.

From the vantage point of numerical analysis, the RIP and bi-Lipschitz embeddings guarantee a very strong property: every reconstruction problem instance is \textit{well conditioned}. Informally, an instance of a computational problem locally modeled by a function is called \textit{well conditioned} if the condition number is small, e.g., $\kappa \approx 1$, and \textit{ill conditioned} if the condition number is large, e.g., the inverse of typical backward errors. If $\mu$ is a bi-Lipschitz map, then one has a global upper bound on the condition numbers of both $\mu$ and $\mu^{-1}$. We include the precise statement of this result for completeness; see \cite[Proposition 6.3.10]{HJE2017}.

\begin{proposition} \label{prop_RIP_implies_cond_bound}
Let $\mu : \Var{X} \to \Var{Y}$ be a $(c,C)$-bi-Lipschitz map, then we have
$$c \le \kappa[\mu](x) \le C\quad\text{and}\quad\tfrac{1}{C} \le \kappa[\mu^{-1}](\mu(x)) \le \tfrac{1}{c}$$
for all $x$,
where the condition number $\kappa$ is relative to the metrics of $(\Var{X}, \dist_\Var{X})$ and $(\Var{Y},\dist_\Var{Y})$ in \cref{eqn_kappa}.
\end{proposition}

A finite condition number $\kappa$ also implies the existence of the local Lipschitz constant $\kappa + \delta$ for arbitrary $\delta>0$ of $f$ in a neighborhood of $x$.

Instead of relying on the RIP or bi-Lipschitz embeddings to guarantee \textit{every} reconstruction problem instance is well conditioned, we propose to rely on the condition number of specific problem instances to quantify how difficult the reconstruction problem is. For forward, backward, and mixed stable numerical algorithms \cite{Higham1996}, the bound \cref{eqn_stab} then provides a good upper bound on the reconstruction error of a specific problem instance, even if the true solution $x$ is not known.

In the answer to \cref{Q11} below, we show that $\kappa[\mu^{-1}](\mu(x))$ can be computed from the singular values of the derivative of $\mu$. Therefore, it can be approximated numerically for differentiable measurement maps using standard numerical linear algebra algorithms. If an explicit expression of the differential of $\mu$ is unknown, one could even apply numerical differentiation techniques to approximate it. Hence, computing the condition number is often significantly easier than determining global lower and upper Lipschitz constants of $\mu$ (or $\mu^{-1})$.

By relying on the condition number to furnish reconstruction error bounds (for forward stable algorithms) and moving away from the RIP or bi-Lipschitz embeddings, we sacrifice the \textit{global} upper bound on the condition number of $\mu^{-1}$. That is, we allow that some reconstruction problems are very sensitive to tiny perturbations of the input; hence, they are numerically hard or even impossible to solve \cite{BC2013}. Because of \cref{eqn_stab}, this also means that these instances should not be expected to be solved with a small forward error \cite{Armentano}. However, this does not mean that most, typical, or average problem instances are difficult to solve accurately!

\subsection{Summary}
The $\epsilon$-RIP may be difficult or impossible to establish for complicated models~$\Var{X}$. We therefore believe that trading a global bound on the reconstruction error and global identifiabilty for an algorithmically computable local upper bound and corresponding local identifiabilty (recoverability) can often be a worthwhile tradeoff. This viewpoint is developed further in the rest of the paper.

\section{Questions \& Answers }\label{q_and_a}
We answer several individual questions about recoverability, identifiability, and continuity. Combining these in the appropriate ways establishes the well-posedness of an algebraic compressed sensing problem. The proofs for the statements in this section are given in \cref{q_and_a_proofs}. The background in algebraic geometry used in the proofs is introduced in \cref{sec:AG}. The questions and answers, however, can mostly be understood without this background.

Throughout this section, let $\Var{X}\subseteq \bbk^n$, where $\bbk=\RR$ or $\bbk=\CC$, be specified either implicitly as in (IS) or explicitly as in (ES). The definition of these models in the complex case is obtained by replacing $\RR$ by $\CC$ in both settings. Our answers are general: they do not rely on specific properties of a particular model.

The following discussion centers around the numbers $d,s>0$. The first is the dimension of $\Var{X}$ (real dimension if $\bbk=\mathbb R$, and complex dimension if $\bbk=\mathbb C$). This number $d$ is fixed, as the model $\Var{X}$ specifying the signal structure is assumed to be fixed beforehand. The number $s$ is the number of measurements and can be varied through the design of the measurement map  \(
\mu : \Var{X}\to \bbk^s
\) in \cref{eqn_ip}.
Of particular interest is the case where $\mu$ is the restriction to $\calX$ of a linear or an affine linear map $\hat{\mu}:\bbk^n\rightarrow \bbk^s$. An affine linear map has the form $\hat{\mu}(x)=Ax+b$, where $A\in\bbk^{s\times n}$ and $b\in \bbk^s$. Since shifting by a vector is an isomorphism of $\bbk^s$ (as an affine variety), geometric properties of the map $\hat{\mu}$ are directly related to the analogous properties of the \emph{linear map} $x\mapsto Ax$. For this reason, we restrict the analysis to linear maps. Some of the questions below concern the case where $\mu$ is a polynomial map.

We first discuss existence. This consists of deciding whether a given $y \in \bbk^s$ sits in the image of the model $\Var{X}$ under the measurement map $\mu$. Next, we consider recoverability, which consists of determining under which conditions the preimage $\mu^{-1}(\mu(x))$ is finite for a point $x \in \calX$. Then, we discuss identifiability, which consists of determining when the preimage is a single point. Finally, we investigate the local continuity of the elements in the fiber of $\mu$.

Many of the statements below are formulated for ``generic'' data points. This concept is explained formally in \cref{subsec: topology} below. In particular, in all the questions below, generic points form an open dense subset of all points.

\subsection{Existence}

The first question concerns the existence of solutions.

\begin{question}\label{QExistence}
 Under what conditions do there exist solutions $y = \mu(x)$ for a given point~$y \in \bbk^s$?
\end{question}
\begin{answer}
First, suppose that $\Var X$ is the zero locus of the polynomials $f_1,\dots,f_k$ as in (IS). Then for a fixed measurement $y=(y_1,\dots,y_s)$ the equality $y = \mu(x)$ holds for some $x\in \Var X$ if and only if the following system of $k+s$ polynomial equations in the variables $x_1 \vvirg x_n$ has a solution:
\begin{align*}
f_1(x) = \cdots =f_k(x) = y_1- \mu_1(x) = \cdots =  y_s- \mu_s(x) = 0.
\end{align*}
Now suppose that $\Var X$ is given as image of a polynomial parameterization as in (ES). Then for a fixed $y$, the equality $y = \mu(x)$ holds for some $x\in X$ if and only if the following system of $s$ polynomial equations in the variables $t_1 \vvirg t_\ell$ has a solution:
\begin{align*}
y_1 - \mu_1(\phi_1(t),\ldots,\phi_n(t)) = \cdots =y_s - \mu_s(\phi_1(t),\ldots,\phi_n(t))= 0.
\end{align*}

The existence of solutions of a polynomial system $g_1 = \cdots = g_N = 0$ as above is characterized as follows.

For $\bbk =\CC$, the system has solutions if and only if the ideal $\langle g_i : i =1 \vvirg N \rangle$ does not contain $1$. Explicitly, this condition can be tested using Gr\"obner basis methods \cite{whatisGrobner,CLO2015,SingularCommAlg,sturmfels1996grobner,Sturmfels02solvingsystems}, which are implemented in most computer algebra systems.

For $\bbk = \bbR$, the system has a solution if and only if the \emph{real radical}, in the sense of \cite[Section 4.1]{BCR1998}, does not contain $1$.
\end{answer}

We note that neither for $\bbk =\RR$ nor for $\bbk =\CC$ it is enough to check if some polynomials vanish only at $y$ to guarantee the existence of $x\in \Var X$ with $y=\mu(x)$. The reason is that a projection of a variety is not necessarily a variety. However, a consequence of Chevalley's Theorem (see \cref{chevalley}) is that, in the case $\bbk =\CC$, there exists a finite number of polynomials $\{h_i(y) : i = 1 \vvirg M\}$, such that the vanishing or non-vanishing of $h_i(y)$ for a fixed $y$ determines whether there is $x \in \calX$ with $y = \mu(x)$. Algorithms providing these $h_i$ were implemented in \cite{harris2019computing, barakat2019algorithmic}. For $\bbk =\RR$ one needs to check a finite number of polynomial equations and polynomial \emph{inequalities}; an implementation is provided in the Maple package \texttt{RegularChains}.

While the above answer is necessary and sufficient, it is substantially less practical than all of the Q\&As that will follow. In fact, deciding whether a solution to a system of polynomial equations exists is an NP-hard decision problem, see, e.g., \cite{FraYes:ComplexitySolvingAlgebraicEquations}. Nevertheless, one can resort to numerical methods such as certified homotopy continuation \cite{BL2012,HS2012}, which can find approximate numerical roots. In practice, uncertified homotopy continuation can be a faster alternative with similar accuracy in many cases of practical interest \cite{BT2018,Lee:Li:Tsai:2008,Bertini,NumericalAlgebraicGeometryArticle,Verschelde:PHCpack,TvBV2020}. Lairez \cite{Lairez2017,Lairez2020} proved that finding one approximate root of (certain) random polynomial systems has a quasi-linear time complexity in the Blum--Shub--Smale model of computation \cite{BSS1989}.

\subsection{Recoverability}
We start with a remarkable, fundamental result about recoverability of generic linear measurement maps.
We call a map $\mu : \Var{X} \to \bbk^s$ \textit{recoverable at a point} $x \in \Var{X}$ if the fiber $\mu^{-1}(\mu(x))$ is finite. A map is \textit{everywhere recoverable} if it is recoverable at every point $x\in \Var{X}$. A map is \emph{generically recoverable} if it is recoverable for generic points $x\in \Var{X}$.

\begin{question}\label{question:finite map}
Under what conditions is a generic affine linear map $\mu: \Var{X}\rightarrow \bbk^s$ everywhere recoverable?
\end{question}
\begin{answer}
A generic affine linear map is everywhere recoverable if and only if the number of measurements is $s\geq d$.

In fact, if $\bbk = \bbC$ and $s=d$, the number of solutions of \cref{eqn_ip} at a generic problem instance $y \in \mu(\Var{X})$ equals a constant $\delta$, called the \textit{degree of $\Var{X}$}, see \cref{sec:degree}. In the real case (i.e., $\bbk = \bbR$), the number of real solutions for a generic problem instance is bounded above by that same $\delta$.
\end{answer}

The previous Q\&A can be extended to certain polynomial maps.
\begin{question}\label{question:finite map2}
Under what conditions is a measurement map $\mu: \Var{X}\rightarrow \bbk^s$ defined by $s$ generic (non-homogeneous) polynomials of degree $p$ everywhere recoverable?
\end{question}
\begin{answer}
A map defined by $s$ generic polynomials of degree equal to $p$ is everywhere recoverable if and only if $s\geq d$.

For $\bbk = \bbC$ and $s=d$, the number of solutions of \cref{eqn_ip} at a generic problem instance $y \in \mu(\Var{X})$ equals $\delta \cdot p^d$ where $\delta$ is the degree of $\Var{X}$. If $\bbk =\bbR$, the number of real solutions for a generic problem instance is bounded from above by~$\delta \cdot p^d$.
\end{answer}

The answer to \cref{question:finite map2} is implied by applying the \emph{Veronese embedding} to the setting of \cref{question:finite map}. We believe the answer to \cref{question:finite map2} to hold in the more general case where $\mu$ is a generic polynomial map with components of possibly different degrees. We expect that a proof should follow by adapting the proof of Noether's Normalization Lemma (\cref{thm:Noether Norm}) to this general setting. We leave this for future work.

\begin{example}[Random maps]
The space of linear maps and the space of polynomial maps with a fixed degree are finite dimensional vector spaces. One can randomly sample such maps from a probability distribution that is absolutely continuous with respect to the Lebesgue measure. The answers to \cref{question:finite map,question:finite map2} guarantee that such randomly sampled maps are recoverable with probability $1$.
This includes all linear and polynomial maps whose coefficients are sampled identically and independently distributed from any continuous probability distribution, such as \texttt{randn} and \texttt{rand} in Julia, Matlab, Octave, and Python.

There is a vast body of literature on robust recovery from random measurements. For instance, the case of low-rank matrix recovery from Gaussian measurements is discussed in \cite{8204873}: the authors show that $10r(m+n)$ Gaussian measurements suffice to robustly recover rank-$r$ matrices in $\mathbb R^{m\times n}$ via nuclear norm minimization. However, the dimension of the variety of rank-$r$ matrices is $d=r(m+n-r)$, so the answer to \cref{question:finite map} guarantees that $s=d$ measurements are enough for recoverability. This provides a significant improvement over \cite{8204873}, albeit with only local guarantees on the robustness of the reconstruction.
\end{example}

In \textit{data completion} or \textit{imputation}, one tries to recover a structured signal $x \in \Var{X}$ from partially observed data, i.e., a subset of the coordinates of $\bbk^n$. In this case, the measurement map is a \textit{coordinate projection}. These are specific linear maps for which the foregoing Q\&A's do not apply a priori. Even randomly selecting the~$s$ coordinates on which to project does not guarantee a generic linear measurement map in the sense of \cref{question:finite map,question:finite map2}. Nevertheless, for coordinate projections we have the following result.
\begin{question}\label{Q_coordinate_recovery}
Under what conditions does there exist a coordinate projection $\hat{\mu}:\bbk^n \to \bbk^s$ that is generically recoverable when restricted to $\Var{X}$?
\end{question}
\begin{answer}
There exists a generically recoverable coordinate projection if and only if $s\geq d$.
\end{answer}

The following example shows that \textit{everywhere} recoverability is not possible in general for coordinate projections. More precisely, the only everywhere recoverable coordinate projection is in general the identity map $\hat{\mu} : \bbk^n \to \bbk^n$.

\begin{example}\label{coord_proj_not_recoverable}
We provide an example of a variety $\calX \subseteq \bbk^n$ with the property that every coordinate projection $\mu: \calX \to \bbk^s$, with $s < n$, has at least one infinite fiber. Consider the parameterization $\phi : \bbk^2 \to \bbk^n$ defined by $\phi(t_1,t_2) = (\phi_1 (t_1,t_2) \vvirg \phi_n(t_1,t_2))$, where
\(
 \phi_j(t_1,t_2) = t_2 \cdot \textprod_{i \neq j} (t_1 - i).
\)
It is easy to show that the Jacobian matrix of $\phi$ has rank $2$ for every $(t_1,t_2)$ with $t_2 \neq 0$. Let $\calX = \bar{\im(\phi)}$, where the overline denotes the closure in the Euclidean topology. This is an algebraic variety of dimension $d = 2$. 
For $j = 1 \vvirg n$, we have
\[
\phi(j,t_2) = t_2 (0 \vvirg 0, \underbrace{\textprod_{i \neq j} (j - i)}_{j\text{-th entry}}, 0 \vvirg 0 ).
\]
In particular, the lines corresponding to the coordinate axes of $\bbk^n$ are contained in $\calX$. Now, every coordinate projection maps at least one coordinate axis to $0$. Therefore, if $\mu: \calX \to \bbk^s$ is a coordinate projection with $s < n$, the fiber $\mu^{-1}(y)$ for~$y = 0 \in \bbk^s$ is always infinite, containing at least one of these lines.
\end{example}

We stress that the answer to \cref{Q_coordinate_recovery} only guarantees existence of a generically recoverable coordinate projection.
We do not claim that most or even several coordinate projections with $s\ge d$ are generically recoverable. In fact, constructions similar to the one of \cref{coord_proj_not_recoverable} provide examples of algebraic varieties with the property that only some specific coordinate projections are generically recoverable. An example of this behavior is presented next.

\begin{example}[Rank-$1$ matrix recovery] \label{example_recov_lr}
Let $\calX$ be the variety of $2 \times 3$ matrices of rank at most $1$, that is $\calX = \phi( \bbk^5)$, with
\begin{equation}\label{eqn: map parameterizing rank one}
\begin{aligned}
\phi: \bbk^5 \to \bbk^{2 \times 3},\quad
(t_1 \vvirg t_5) \mapsto \begin{bmatrix}t_1 \\  t_2 \end{bmatrix} \begin{bmatrix} t_3 & t_4 & t_5\end{bmatrix}.
\end{aligned}
\end{equation}
One can verify that $d = \dim \calX = 4$.
Let $\mu$ be the measurement map defined by restricting the coordinate projection
\begin{align*}
\hat{\mu}: \bbk^{2 \times 3} \to \bbk^{4},\quad
 \begin{bmatrix}
  x_1 & x_2 & x_3 \\
  x_4 & x_5 & x_6
 \end{bmatrix}
\mapsto (x_1,x_2,x_4,x_5)
\end{align*}
to $\Var{X}$.
There is no point $y \in \bbk^4$ at which this coordinate projection is recoverable. Indeed, if $A = \left[\begin{smallmatrix}t_1 \\  t_2 \end{smallmatrix}\right] \left[\begin{smallmatrix} t_3 & t_4 & t_5\end{smallmatrix}\right]$, then
\(
\mu^{-1}(\mu(A)) \supseteq \left\{ \left[\begin{smallmatrix}t_1 \\ t_2 \end{smallmatrix}\right] \left[\begin{smallmatrix} t_3 & t_4 & t' \end{smallmatrix}\right] : t' \in \bbk \right\};
\)
hence, every fiber has dimension (at least) one.

On the other hand, \changed{every} coordinate projection \changed{that forms a ``cross'' as discussed in \cite{BR2000}, like}
\begin{align*}
\tilde{\mu}: \bbk^{2 \times 3} \to \bbk^{4},\quad
 \begin{bmatrix}
  x_1 & x_2 & x_3 \\
  x_4 & x_5 & x_6
 \end{bmatrix}
\mapsto (x_1,x_2,x_3,x_4),
\end{align*}
is generically recoverable on $\Var{X}$. Note that $\tilde{\mu}$ is not everywhere recoverable: if $x_1=x_4 = 0$, the fiber $\mu^{-1}(\mu(A))$ is at least $1$-dimensional. Like \cref{coord_proj_not_recoverable}, there is no everywhere recoverable coordinate projection except for the identity map.

\changed{Conditions under which recoverability and identifiability of low-rank matrix completion problems holds were discussed in \cite{BR2000,SC2010,KTT2015,GHIL:complexity_lin_circuits_and_geometry,Tsakiris2021,Tsakiris2021b}, among others.}
\end{example}

Given a polynomial measurement map, one can explicitly determine whether it is generically recoverable by computing the rank of the differential of a suitable map. This is discussed next.

\begin{question}\label{question: specific map is recoverable}
Given a specific polynomial map $\mu : \calX \to \bbk^s$, how can one determine whether~$\mu$ is generically recoverable?
\end{question}
\begin{answer}
First consider the implicit setting (IS) and let the variety be given as $\Var{X} = \{ x \in \bbk^n : f_1(x) = \cdots = f_k(x) = 0 \}$. Write $f = (f_1 \vvirg f_k)$, which can be regarded as a polynomial map from $\bbk^n$ to $\bbk^k$. A sufficient condition to guarantee that $\mu$ is generically recoverable is that there exists a point $x \in \calX$ satisfying
\[
\ker (\deriv{\mu}{x}) \cap \ker (\deriv{f}{x}) = 0 .
\]
Explicitly, regarding $\deriv{\mu}{x}$ as an $s \times n$ matrix and $\deriv{f}{x}$ as a $k \times n$ matrix, this condition is equivalent to the $(s+k) \times n$ matrix $\left[ \begin{smallmatrix} \deriv{\mu}{x} \\ \deriv{f}{x} \end{smallmatrix}\right]$ having rank $n$.

Now consider the explicit setting (ES), and let $\calX$ be parameterized explicitly by $\phi = (\phi_1 \vvirg \phi_n) : \bbk ^\ell \to \bbk^n$.  Then a sufficient condition to guarantee that $\mu$ is generically recoverable is that there exists a point $t \in \bbk^\ell$ such that the Jacobian matrix of $\mu \circ \phi : \bbk^\ell \to \bbk^s$ at $t$ has rank $d = \dim \Var{X}$.
\end{answer}

\begin{remark} There may be several ways to present $\Var{X} = \{ x \in \bbk^n : f_1(x) = \cdots = f_k(x) = 0 \}$ and even the ideal $(f_1,\dots,f_k)$ generated by the $f_i$'s may change. It is always best to choose the $f_i$'s that generate the whole ideal of polynomial functions that vanish on $\Var{X}$.
\end{remark}
\begin{remark}
We note that the condition that a map \changed{$\mu : \Var{X} \to \bbk^s$} is generically recoverable is equivalent to the fact that the dimension of the image \changed{$\mu(\Var{X})$} is equal to $d$.
\end{remark}

The second criterion is especially convenient for low-dimensional varieties embedded in a high-dimensional ambient space. In this case, it suffices to compute the rank of an $\ell \times s$ matrix, where $\ell$ is the number of parameters in the parameterization of $\Var{X}$ and $s$ the number of measurements, both of which can be similar in size to $d =\dim \Var{X}$. In a numerical setting, computing the rank consists of counting the number of significant singular values.

The following question addresses explicitly the case of coordinate projection, in the setting of parameterized models.

\begin{question}\label{Q_which_coordinate_projection}
If $\Var{X}$ is explicitly parameterized by $\phi: \bbk^\ell \to \bbk^n$ as in (ES), how can one determine a subset of $s \geq d$ coordinates such that the coordinate projection on these $s$ coordinates is generically recoverable?
\end{question}
\begin{answer}
The $s$ distinct coordinates $i_1 \vvirg i_s$ should be selected so that the $s \times \ell$ matrix $\left[ \frac{\partial \phi_{i_p}}{\partial t_j}\right]_{\substack{p = 1 \vvirg s \\ j = 1 \vvirg \ell}}$ has rank exactly $d$ for at least one $t \in \bbk^\ell$.
\end{answer}

We can now use this answer to understand \cref{example_recov_lr}.

\begin{example}[Rank-$1$ matrix recovery, continued I]\label{ex_rank1cont}
Write $\phi_{ij}$ for the components of the map $\phi$ of \eqref{eqn: map parameterizing rank one}. In \cref{example_recov_lr}, we claimed that the coordinate projection defined by $\hat{\mu} ( \left[\begin{smallmatrix} x_1 & x_2 & x_3 \\ x_4 & x_5 & x_6 \end{smallmatrix}\right] ) = (x_1,x_2,x_4,x_5)$ is not recoverable. We verify that the sufficient condition from \cref{question: specific map is recoverable} is not satisfied. Indeed, the Jacobian matrix of $\mu \circ \phi : (t_1 \vvirg t_5) \mapsto (t_1t_3,t_1t_4,t_2t_3,t_2t_4)$ is
\[
J = \deriv{(\mu \circ \phi)}{t} = \begin{blockarray}{cccccc}
  & \partial_1& \partial_2 & \partial_3 & \partial_4 & \partial_5 \\
  \begin{block}{c[ccccc]}
 \phi_{1,1} & t_3 & 0 & t_1 & 0 & 0 \\
 \phi_{1,2} & t_4 & 0 & 0 & t_1& 0 \\
 \phi_{2,1} & 0 & t_3 & t_2 & 0 & 0 \\
 \phi_{2,2} & 0 & t_4& 0 & t_2 & 0 \\
  \end{block}
 \end{blockarray},
\]
which has rank at most $3$ for every choice of $(t_1 \vvirg t_5)$ \changed{because $t_1 \partial_1 + t_2 \partial_2 = t_3 \partial_3 + t_4 \partial_4$}.

On the other hand, the coordinate projection $\tilde{\mu}$ is generically recoverable. In this case $\tilde{\mu} \circ \phi (t_1 \vvirg t_5) = (t_1t_3,t_1t_4,t_1t_5,t_2t_3)$ and its Jacobian matrix is
\begin{equation}\label{jacobian_f}
\tilde{J} = \deriv{(\tilde{\mu}\circ \phi)}{t} = \begin{blockarray}{cccccc}
  & \partial_1 & \partial_2 & \partial_3 & \partial_4 & \partial_5 \\
  \begin{block}{c[ccccc]}
 \phi_{1,1} & t_3 & 0 & t_1 & 0 & 0 \\
 \phi_{1,2} & t_4 & 0 & 0 & t_1 & 0 \\
 \phi_{1,3} & t_5 & 0 & 0 & 0 & t_1 \\
 \phi_{2,1} & 0 & t_3 & t_2 & 0 & 0 \\
  \end{block}
 \end{blockarray}.
\end{equation}
For $t_1=t_3=1$ and $t_2=t_4=t_5=0$, the $\rank(\tilde{J}) = 4$. Hence, $\tilde{\mu}$ is generically recoverable by \cref{question: specific map is recoverable}.
\end{example}

\subsection{Identifiability}
Recall that a map $\mu:\calX\to\bbk^s$ is \textit{identifiable at a point} $x \in \Var{X}$ if the fiber $\mu^{-1}(\mu(x))=\{x\}$ consists of a single point. We say that $\mu$ is \emph{everywhere identifiable} if it is identifiable at every $x \in \Var{X}$ and that it is \textit{generically identifiable} if it is identifiable at generic points $x\in \Var{X}$.

In this section we discuss three settings. The first one considers identifiability for generic maps and generic points on $\Var{X}$. This is the same setting as the one discussed in the case of recoverability in the previous subsection.

\begin{question}\label{Q3}
Under what conditions is a generic linear map $\mu: \Var{X}\rightarrow \bbk^s$ generically identifiable?
\end{question}
\begin{answer}
If $\calX$ is a linear space, then $\mu$ is generically identifiable if and only if $s \geq d$.  If $\calX$ is not a linear subspace and if $s \geq d+1$, then $\mu$ is generically identifiable. If $\bbk = \bbC$, then the condition $s \geq d+1$ is also necessary.
\end{answer}

The analogous question for polynomial maps has a similar answer.
\begin{question}\label{Q4}
Under what conditions is a map $\mu: \Var{X}\rightarrow \bbk^s$ defined by $s$ generic polynomials of degree~$p$ generically identifiable?
\end{question}
\begin{answer}
Such a polynomial map is generically identifiable if $s\geq d+1$. When $\bbk=\CC$, $d>0$, and $p>1$, the condition $s \geq d+1$ is necessary as well.
\end{answer}

Now, we move from questions about generic data to questions about properties that hold \emph{everywhere}. In \cref{ques:inj} we wonder when a generic linear map is everywhere identifiable and \cref{ques:inj1} asks for the existence of a Euclidean open set of such maps.

A key role in these questions is played by the \emph{set of differences}
\[
\Delta(\Var{X}) := \{x_1-x_2  :  x_1,x_2\in \Var{X}\} \subseteq \bbk^n.
\]
This is the image of the map $\Var{X} \times \Var{X} \to\bbk^n, (x_1,x_2)\mapsto x_1-x_2$. The closure of $\Delta(\calX)$ is an algebraic variety of dimension at most $2d$. When $\Var{X}$ is a cone over $0 \in \bbk^n$, then $\Delta(\Var{X})$ is the \textit{Minkowski sum} of $\Var{X}$ with itself, and its closure is called the \emph{secant variety} of $\Var{X}$ \cite{Harris1992,Russo:GeometrySpecialProjVars}.

One can immediately observe that for a linear map $\hat{\mu}: \bbk^n \to \bbk^s$, one has
\[
 \hat{\mu}(x_1) = \hat{\mu}(x_2) \quad\text{if and only if}\quad \hat{\mu}(x_1 - x_2) = 0.
\]
In other words, $\mu=\hat{\mu}|_\calX$ is everywhere identifiable if and only if the kernel $\ker(\hat{\mu})$ intersects $\Delta(\Var{X})$ only at $0 \in \bbk^n$.

\begin{question}\label{ques:inj}
Under which conditions is a generic linear map $\mu: \Var{X}\rightarrow \bbk^s$ everywhere identifiable?
\end{question}
\begin{answer}\label{answ:inj}
A generic linear map $\mu$ is everywhere identifiable if $s \ge 1+\dim \Delta(\Var{X})$. In particular, if $s \ge 2d+1$, $\mu$ is everywhere identifiable.
\end{answer}

We provide two interesting examples; one where the answer to \cref{ques:inj} is not optimal, and one where it is.

\begin{example}[Nonoptimality of \cref{answ:inj} for cones]
Suppose $\Var{X}$ is a cone over $0 \in \bbk^n$: for every $x\in \Var{X}$ and every $t\in \bbk$, one has $tx\in\Var{X}$. Then, $\Delta(\Var{X})$ is also a cone over $0$; let $d'$ be its dimension. In this case, a generic projection to~$\bbk^{d'}$ is injective on $\calX$, and therefore everywhere identifiable.
\end{example}
\begin{example}[Optimality of \cref{answ:inj} for the twisted cubic]\label{exam: twisted cubic}
Let $\Var{X}$ be the curve in~$\CC^3$ defined as the image of $\phi:\CC\to\CC^3, t\mapsto (t,t^2,t^3)$. In classic algebraic geometry this is called the \emph{twisted cubic curve} in $\bbC^3$. The set of differences is
\[
\Delta(\calX) = \{ (t_1-t_2, t_1^2 - t_2^2, t_1^3-t_2^3) : t_1,t_2 \in \bbC\}.
\]
 In this case, a generic line through $0$ intersects $\Delta(\calX)$ in a point different from $0$. Therefore, a generic linear map $\hat{\mu}: \bbC^3 \to \bbC^2$ satisfies $\ker(\hat{\mu}) \cap \Delta(\calX) \neq 0$. This determines a point (in fact at least two points) $x \in \calX$ at which $\mu$ is not identifiable. Hence, the bound $s \geq \dim \Delta(\calX) + 1 = 3$ of the answer to \cref{ques:inj} is optimal.

There can be special projections with $s < \dim \Delta(\calX) +1$ taking $\calX$ injectively to~$\bbC^s$, even if generic projections are not injective. For instance, in the case of the twisted cubic, the projection $(x_1,x_2,x_3) \mapsto x_1$ maps~$\calX$ injectively to $\bbC^1$.
\end{example}

The precise meaning of \emph{generic} is given in terms of open subsets in the Zariski topology, as explained in Section \ref{sec:AG}. Over the complex numbers, these open sets are dense, with respect to the classic Euclidean topology, in their Zariski closure. As a consequence, requiring a property to hold generically is usually a strong condition. It is therefore natural to wonder whether one can relax the genericity condition, requiring that a property holds on a set containing a Euclidean open subset. Over the real numbers, it is often possible to determine Euclidean open subsets of linear projections $\hat{\mu} : \bbR^n \to \bbR^s$ which are identifiable (over the reals) even though the corresponding extension $\hat{\mu}_\bbC : \bbC^n \to \bbC^s$ is not. However, the existence of such open sets of projections depends on the geometry of the model, and not only on its dimension, as pointed out in the next question.

\begin{question}\label{ques:inj1}
Under what conditions does there exist a Euclidean open set of everywhere identifiable linear maps $\mu: \Var{X}\rightarrow \RR^s$?
\end{question}
\begin{answer}
There is an open set of everywhere identifiable linear maps if we have $s\geq 1+\dim \Delta(\Var{X})$. In particular, this holds if $s \ge 2d+1$.
\end{answer}

We provide an example of a model $\calX$ for which there exists an open set of identifiable linear maps $\mu: \calX \to \bbR^s$ for $s < 1+\dim(\Delta(\Var X))$.

\begin{example}[The twisted cubic reconsidered]\label{example_real_twisted_cubic}
Let $\calX$ be the twisted cubic from Example \ref{exam: twisted cubic}. Its set of differences over the reals is
\[
\Delta(\Var{X})=\{(t_1-t_2,t_1^2-t_2^2,t_1^3-t_2^3)  :  t_1,t_2\in \RR\}.
\]
In particular, if $(x_1,x_2,x_3) \in \Delta(\calX)$, then $x_1$ and $x_3$ have the same sign. As a consequence, any linear map $\hat{\mu}: \bbR^3 \to \bbR^2$ whose kernel is the line $\ker \hat{\mu} = \langle (1,\alpha,\beta) \rangle$ for some $\beta < 0$ satisfies $\ker(\hat{\mu}) \cap \Delta(\calX) = 0$, so that $\mu$ is everywhere identifiable. This is a Euclidean open set of linear maps.
\end{example}

Finally, we return to the case of coordinate projections from data completion applications. We stress that it is not possible in general to guarantee identifiability results in this special setting. This is highlighted in the following example.

\begin{example}\label{exam: even powers non identifiable}
 Let $\calX$ be an algebraic variety defined implicitly as the vanishing locus of polynomials $f_1 \vvirg f_k$ and assume $\calX$ does not lie in a coordinate hyperplane. Suppose that each variable $x_i$ appears in $f_1 \vvirg f_k$ only with even powers. In this case, if $( \xi_1 \vvirg \xi_n)$ is a point of $\calX$, then any point of the form $(\pm \xi_1 \vvirg \pm \xi_n)$ belongs to $\calX$ as well. Consequently, if $\hat{\mu} : \bbk^n \to \bbk^s$ is a coordinate projection different from the identity, then $\mu=\hat{\mu}|_\calX$ is not generically identifiable. For instance, if $x = (\xi_1 \vvirg \xi_n) \in \calX$ is a generic point, and $\hat{\mu}$ is the map that projects away the last coordinate, then $\mu^{-1}(\mu(x)) = \{ (\xi_1 \vvirg \xi_{n-1},\pm \xi_n)\}$. Note that $\xi_n \ne 0$ because of the genericity condition and the fact $\calX$ does not lie in a coordinate hyperplane.

 A minimal example is the following:
 \[
  \calX = \left\{ (x_1,x_2,x_3) \in \bbk^3 : x_1^2 + x_2^2 + x_3^2 - 1 = x_1^2x_2^2-x_3^4-1 = 0 \right\}.
 \]
In this case, the variety $\calX \subseteq \bbk^3$ has dimension $1$ and does not admit a generically identifiable coordinate projection. On the other hand, a generic projection to $\bbk^2$ is generically identifiable.
\end{example}

A consequence of \cref{exam: even powers non identifiable} is that in general the only everywhere identifiable coordinate projection is the identity map. This also follows from the analogous result on recoverability, as seen in \cref{coord_proj_not_recoverable}.

\begin{example}[The rational normal curve]
The $1$-dimensional rational normal curve is again an interesting example:
\(
 \mathcal X = \{ (t, t^2, t^3, \dots, t^n)  :  t \in \bbk \} \subset \bbk^n.
\)
The answer to \cref{ques:inj} ensures there is an identifiable projection to $1 + \dim \Delta(\calX) = 3$ coordinates. For example, projecting to the first three coordinates $(t, t^2, t^3)$ enables us to recover $t$ (from the first element alone). In fact, this shows once more that the answer to \cref{ques:inj} is not optimal: a projection to the first coordinate of $\bbk^n$ is already everywhere identifiable.

On the other hand, there are several generically nonidentifiable coordinate projections, even using much more than $1 + \dim\Delta(\calX)=3$ coordinates. Consider for example projecting to any subset of $3 \le s \le \frac{n}{2}$ even coordinates. This leaves at least two possibilities for the reconstruction: If $(t_0,t_0^2,t_0^3,\ldots,t_0^n) \in \calX$ is one solution, then necessarily $(-t_0,t_0^2,-t_0^3,\ldots,(-t_0)^n)$ is another solution.
\end{example}

We conclude the discussion of identifiability emphasizing the intricacy of the questions in this section, even in the mathematically friendlier setting of complex projective geometry. For example, it is an open problem whether a $d$-dimensional complex projective variety may be algebraically injected to a projective space of dimension $2d$ using arbitrary algebraic maps; see \cite{BV, Paul2, Paul1} and \cite[Question 3]{jaOber}. A weaker negative result is known:
\changed{there are algebraic curves $\calX$ in a complex $3$-dimensional projective space that cannot be injectively projected to the projective plane \cite{Piene}. In our affine setting this corresponds to an algebraic surface ($d=2$), inside a $4$-dimensional space ($n=4$) so that for any projection to a plane ($s=2$) there exist two lines inside the surface that are mapped to the same line.}

Finally, we point out that the structure described in \cref{exam: even powers non identifiable} is very special and in most cases we expect a generically identifiable coordinate projection $\bbk^n \to \bbk^s$ to exist for $s = d+1$, as in the case of generic projections. We sketch an argument we believe could be formalized to guarantee the existence of such projection. Recall that from the answer to \cref{Q_coordinate_recovery}, there exists a generically recoverable coordinate projection to $d$ coordinates. Let $\mu: \calX \to \bbk^d$ be such projection. A classic construction in algebraic geometry is the \emph{second fiber power of $\calX$ over $\mu$}, denoted $\calX^{\mu[2]}$. Informally, this is an algebraic variety containing pairs of points having the same image under $\mu$. The variety $\calX^{\mu[2]}$ is reducible; its irreducible components of highest dimension have dimension $d$ and there are at least two of them. One can prove that if the number of $d$-dimensional components of $\calX^{\mu[2]}$ is exactly two, then there is a coordinate projection ${\mu}' : \calX \to \bbk^{d+1}$ which is generically identifiable; in fact, ${\mu}'$ factors $\mu$, in the sense that $\mu$ projects away all the coordinates projected away by ${\mu}'$ plus an additional one.

\subsection{Continuity}\label{sec_continuity}
The answers from the previous subsections give several conditions under which the measurement map is locally or globally injective. In particular, under these conditions, the measurement map admits a (local) inverse. The final requirement for an inverse problem to be well posed is that this (local) inverse map is at least continuous. In fact, much stronger properties hold.

First, we recall terminology to succinctly ask the questions. We refer to \cite{Spivak1965,Lee2013} for the theory and definitions. A map between manifolds is called \textit{smooth} if it can be differentiated arbitrarily many times; that is, if it is $C^\infty$ on the local charts defining the manifolds. Smooth maps are continuous and locally Lipschitz continuous.

Let $\calM$ be a smooth $\bbk$-manifold, and let $\mu : \Var{M} \to \bbk^s$ be a map. The map $\mu$ is a \textit{global diffeomorphism onto its image} if $\mu(\Var{M})$ is a smooth manifold, $\mu$ is smooth and injective, and $\mu^{-1}$ is smooth \cite{Lee2013}. The map $\mu$ is a \textit{local diffeomorphism onto its image} at $x \in \Var{M}$ if there is a Euclidean open neighborhood $U_x \subseteq \Var{M}$ of $x$ such that the restriction $\mu|_{U_x}$ is a global diffeomorphism onto its image and $\mu(U_x)$ is open in~$\mu(\calM)$ \cite{Lee2013}.

We investigate which measurement maps $\mu$ are local or global diffeomorphisms.

\begin{question}\label{Q8}
Under what conditions does a generic linear map $\mu: \Var{X}\rightarrow \bbk^s$ have the property that there exists a Zariski open manifold $\Var X'\subseteq \Var X$ such that $\mu(\Var X')$ is a smoothly embedded submanifold of $\bbk^s$ and $\mu|_{\Var X'}$ is a local diffeomorphism onto its image at every point $x\in\Var X'$?
\end{question}

\begin{answer}
There exists such a manifold $\calX'$ if $s \geq d$. Moreover, the bound is sharp, independently of the variety $\calX$.
\end{answer}

\begin{question}\label{Q9}
Under what conditions does a generic linear map $\mu: \Var{X}\rightarrow \bbk^s$ have the property that there exists a Zariski open manifold $\Var X'\subset \Var X$ such that $\mu(\Var X')$ is a smoothly embedded submanifold of $\bbk^s$ and $\mu|_{\Var X'}$ is a global diffeomorphism onto its image?
\end{question}

\begin{answer}
There exists such a manifold $\calX'$ if $s \geq d+1$. If $\calX$ is not an affine space, and if $\bbk = \bbC$, the bound is sharp independently of the variety $\calX$.
\end{answer}

The fact that in both these questions $\mu(\Var X')$ is a smoothly embedded submanifold of $\bbk^s$ is important when solving inverse problems using methods from Riemannian optimization \cite{AMS2008}.

The proofs to the answers of \cref{Q8,Q9} reveal the structure of $\mathcal X'$. In the setting of \cref{Q8}, the manifold $\calX'$ has the property that $\calX \setminus \calX'$ contains the singular locus of $\mathcal X$, the locus of points $x\in\mathcal X$ with $(T_x\mathcal X)\cap \ker(\mu) \neq 0$, and the preimage of the singular locus of $\mu(\calX)$. In the setting of \cref{Q9}, $\calX \setminus \calX'$ additionally contains the set of points $x\in\mathcal X$ on which $\mu$ is not injective.

\begin{example}[Rank-1 matrix recovery, continued II]
Let $\calX$ be the variety of $2 \times 3$ matrices of rank one, parameterized as in \cref{eqn: map parameterizing rank one}.
In \cref{ex_rank1cont}, we proved that the coordinate projection $\tilde{\mu}$ defined by $\tilde{\mu}\left(\left[\begin{smallmatrix} x_1 & x_2 & x_3 \\ x_4 & x_5 & x_6 \end{smallmatrix}\right]
\right) = (x_1,x_2,x_3,x_4)$ is generically recoverable.
Let $\calX' = \calX \cap \{ x_1 \neq 0\}$, which is Zariski open in $\calX$. Then, a point $y =(x_1 \vvirg x_4) \in \mu(\calX) \subseteq \bbk^4$ has a unique preimage
\[
 x = (\mu|_{\calX'})^{-1} (x_1 \vvirg x_4) =
 \begin{bmatrix}
  x_1 & x_2 & x_3 \\
  x_4 & x_4 x_1^{-1} x_2 & x_4 x_1^{-1} x_3
 \end{bmatrix}.
\]
This shows that $\mu|_{\calX'}$ is smooth and injective. The argument of Example \ref{ex_rank1cont} shows that $\deriv{\tilde{\mu}}{x} : T_x \calX \to \bbk^4$ is a linear isomorphism for every $x \in \calX'$. Consequently, the image $\tilde{\mu}(\Var{X}')$ is an immersed submanifold of $\bbk^4$ by \cite[Proposition 5.18]{Lee2013} and $\tilde{\mu}|_{\Var X'}$ is a global diffeomorphism between~$\Var{X}'$ and $\tilde{\mu}(\Var{X}')$.

Note that even though $s = 4 = \dim \calX$, this example does not contradict the sharpness result in the answer to \cref{Q9} because the considered coordinate projection is not generic.
\end{example}

The inverse problem \cref{eqn_ip} is well posed for generic $x \in \Var{X}$ and generic choice of linear projection $\mu : \calX \to \bbk^s$ with $s \geq d+1$ because of the answer to \cref{Q9}. We argued in \cref{sec:review} that computing a global Lipschitz constant for $\mu^{-1}$ is a difficult problem. Instead we proposed to compute the smallest local Lipschitz constant, that is, the condition number of $\mu^{-1}$.

\begin{question}\label{Q11}
How does one compute the condition number of the inverse problem \cref{eqn_ip} at a smooth point $x \in \Var{X}$?
\end{question}
\begin{answer}\label{A11}
The condition number with respect to the Euclidean metrics on both domain and codomain at a smooth point $x\in \Var{X}$ is $\kappa(x) = \frac{1}{ \sigma_{d}( \deriv{\mu}{x} ) }$, where $\sigma_d( \deriv{\mu}{x} )$ is the $d$th largest singular value of the derivative of $\mu : \calX \to \bbk^s$ at $x$.

More concretely, let $M \in \bbk^{s \times n}$ be the matrix representing the linear map $\hat{\mu} : \bbk^n \to \bbk^s$ in the ambient space. Let $Q \in \bbk^{n\times d}$ be a matrix whose columns form an orthonormal basis of the tangent space $\Tang{x}{\Var{X}}$. Then $\kappa(x) = \frac{1}{\sigma_d(MQ)}$.
\end{answer}

\begin{example}[Rank-$1$ matrix recovery, continued III]
Let $\calX$ be the variety $2 \times 3$ matrices of rank one.
The tangent space of $\Var X = \mathrm{im}(\phi)$ at a point $X= \phi(t)$ is spanned by the columns of the Jacobian matrix of $\phi$,
\[
J = \deriv{\phi}{t} = \begin{blockarray}{cccccc}
  & \partial_1 & \partial_2 & \partial_3 & \partial_4 & \partial_5 \\
  \begin{block}{c[ccccc]}
 \phi_{1,1} & t_3 & 0 & t_1 & 0 & 0 \\
 \phi_{1,2} & t_4 & 0 & 0 & t_1 & 0 \\
  \phi_{1,3} & t_5 & 0 & 0 & 0 & t_1 \\
 \phi_{2,1} & 0 & t_3 & t_2 & 0 & 0 \\
 \phi_{2,2} & 0 & t_4& 0 & t_2 & 0 \\
  \phi_{2,3} & 0 & t_5 & 0 & 0 & t_2 \\
  \end{block}
 \end{blockarray}.
\]
Using, for example, a rank-revealing $QR$-decomposition or the singular value decomposition of $J$, one determines a matrix $Q$ whose columns form an orthonormal basis of the column span of $J$. This is the matrix $Q$ required in \cref{A11}.

For instance, consider $X = \left[\begin{smallmatrix} 1 & 1 & 1 \\ 2 & 2 & 2 \end{smallmatrix}\right]$, and $t=(1,2,1,1,1) \in \phi^{-1}(X)$. Then, the matrix $Q$ can be computed numerically via a (truncated) singular value decomposition of $J$:
\[
 Q \approx
 \begin{bmatrix}
  -0.258199 & -0.329793 & -0.156748 &  0.516398 \\
  -0.258199 &  0.300644 & -0.207235 &  0.516398 \\
  -0.258199 &  0.029148 &  0.363983 &  0.516398 \\
  -0.516398 & -0.659586 & -0.313497 & -0.258199 \\
  -0.516398 &  0.601289 & -0.414469 & -0.258199 \\
  -0.516398 &  0.058297 &  0.727966 & -0.258199
 \end{bmatrix}.
\]

Let $M$ be the matrix representing the coordinate projection $\tilde{\mu}$ of \cref{example_recov_lr}.
Then, the condition number of the inverse problem at the (locally unique) solution $X$ is~$\sigma_4( M Q )^{-1} \approx 3.8730$.

On the other hand, the coordinate projection $\hat{\mu}$ of \cref{example_recov_lr} is not recoverable, or equivalently not locally identifiable. The fibers of $\hat{\mu}$ have positive dimension, so that from \cref{eqn_kappa} the condition number is infinite. This can be verified numerically. Composing $Q$ with the matrix of representing $\hat{\mu}$, we obtain a matrix whose singular values, rounded to two significant digits, are $1.0$, $1.0$, $0.82$, and $6.2 \cdot 10^{-17}$. The numerically computed condition number is thus $1.6 \cdot 10^{16}$.
\end{example}

The last question assumed that we are given as input a data point $y\in\mu(\Var{X})$. If the data point is $y\not\in\mu(\Var{X})$ and if one seeks to solve the least squares minimization problem $\min_{x\in \Var{X}}\tfrac{1}{2}\Vert \mu(x)-y\Vert^2$, then the condition number of this alternative problem is determined both by the condition number of $\mu^{-1}$ and the \emph{curvature} of $\mu(\Var{X})$ at $\mu(x)$; see \cite{BV2020} for more precise statements. For the case where $\Var{X}$ is the variety of low-rank matrices, the condition number of the associated least-squares problem was worked out in detail in \cite{BV2021}.

\section{The main technical tools: Classic algebraic geometry}\label{sec:AG}

In this section, we provide the relevant definitions and preliminaries from classic algebraic geometry, in the real and complex settings. We use them in Section~\ref{q_and_a_proofs} to prove the correctness of the answers given in Section \ref{q_and_a}. We refer to \cite{Harris1992,CLO2015} for a general introduction and to \cite{Eis:CommutativeAlgebra,GrifHar:PrinciplesAlgebraicGeometry} for a more precise account.

Let $\bbk = \bbR$ or $\bbk = \bbC$. An \textit{algebraic variety} is a subset $\Var{X}\subseteq \bbk^n$ defined as the vanishing set of a set of polynomials $f_1 \vvirg f_k$ in $n$ variables $x_1 \vvirg x_n$ with coefficients in $\bbk$. The \emph{ideal}
\[
I^{\bbk}(\Var{X}):= \{f\in\bbk[x_1,\dots,x_n]  :  f(x)=0\text{ for all } x \in \Var{X}\}
\]
is the set of polynomials vanishing on $\Var{X}$. Hilbert's basis theorem states that $I^\bbk(\Var{X})$ is finitely generated, say $I^\bbk(\Var{X}) = \langle g_1 \vvirg g_r \rangle$. Depending on whether $\bbk= \bbR$ or $\bbk = \bbC$, we say that~$\Var{X}$ is a real or complex algebraic variety, and that $I^\bbk(\Var{X})$ is its real or complex ideal.

A variety $\Var{X}$ is \textit{irreducible} if it cannot be expressed as the union of two proper subvarieties. Every variety $\Var{X}$ is an irredundant finite union of irreducible subvarieties, called \emph{irreducible components} of $\Var{X}$.
For an arbitrary subset $\calX \subseteq \bbk^n$, we say that $\calX$ is irreducible if its Zariski closure $\bar{\calX} \subseteq \bbk^n$ (see below) is irreducible.

A \emph{semi-algebraic} set in $\bbR^n$ is a finite union of subsets of algebraic varieties satisfying a set of (strict) polynomial inequalities. In particular, every algebraic variety is a semi-algebraic set whose set of defining inequalities is empty.

\subsection{Topology on algebraic varieties} \label{subsec: topology}
The \textit{Euclidean topology} of $\bbk^n$ induces a topology on every subset $S$ of $\bbk^n$. We refer to this subspace topology on $S$ as the Euclidean topology on $S$.
The set of subvarieties of a variety $\Var{X}$ is closed under taking arbitrary intersections and finite unions; hence, algebraic subvarieties of $\Var{X}$ define the closed subsets of a topology on $\Var{X}$. This is the \emph{Zariski topology} on $\calX$.

The Zariski topology is coarser than the Euclidean topology: if $\Var{X}$ is irreducible and $U$ is a Zariski open subset of $\Var{X}$, then $U$ is open and dense in $\Var{X}$ with respect to the Euclidean topology. The subset $\Var{X} \setminus U$ is a proper subvariety of $\Var{X}$ and it has measure zero with respect to the measure on $\Var{X}$ induced by the Lebesgue measure on $\bbk^n$. Hence, a Zariski open subset $U$ also has full measure on~$\Var{X}$.

Let $\Var{X}$ be a semi-algebraic set in $\bbR^n$ and let $\bar{\Var{X}}$ be its closure in the Zariski topology of $\bbR^n$. Let $\Var{X}^\circ$ be the Euclidean interior of $\Var{X}$ in $\bar{\Var{X}}$. Then $\Var{X}^\circ$ is semi-algebraic as well and $\bar{\Var{X}^\circ} = \bar{\Var{X}}$ \cite[Proposition 2.2.2]{BCR1998}.

It is common in algebraic geometry, as we have done in the foregoing sections, to state results \emph{generically} or in terms of \emph{generic points}.
\begin{definition}[Genericity]
Let $\Var{X} \subseteq \bbk^n$ be an irreducible algebraic variety. A property~$\calP$ holds \emph{generically on $\Var{X}$} if it holds for all points in a Zariski open subset $U \subseteq \Var{X}$. The points of~$U$ are called \emph{generic points} (with respect to the property~$\calP$). The set $\Sigma_\calP = \Var{X} \setminus U$ is a proper subvariety of $\Var{X}$, called the \emph{discriminant} of the property $\calP$.
\end{definition}
The discriminant $\Sigma_\calP$ of a property can be tedious to describe explicitly. Nevertheless, if a property $\calP$ holds generically, then the discriminant $\Sigma_\calP$ exists and it has Lebesgue measure zero.

The space of linear maps  $\bbk^n\to \bbk^s$ is a vector space isomorphic to $\bbk^{n s}$. Similarly, the space of polynomial maps of a fixed degree pattern forms a vector space. Thus, one can state results in terms of generic linear maps, or generic polynomial maps, meaning that the results hold for generic points in the corresponding vector space.

\subsection{Dimension and tangent space of an algebraic variety}\label{subsec: dimension and tangent}
Let $\Var{X} \subseteq \bbk^n$ be an irreducible variety with ideal $I^\bbk(\Var{X}) = \langle g_1 \vvirg g_r \rangle$. Then, there is a Zariski open subset $\Var{X}' \subseteq \Var{X}$ such that the rank of the Jacobian matrix of $g = (g_1 \vvirg g_r)$ is constant on $\Var{X}'$. This set $\Var{X}'$ is called the set of \emph{smooth} (or \emph{non-singular}) points of~$\Var{X}$. The dimension of $\Var{X}$ is defined as
\[
\dim \Var{X} = n - \rank( \deriv{g}{x} ) \quad \text{for any $x\in \Var{X}'$}.
\]
If $\bbk= \bbR$ (respectively, $\bbk = \bbC$), then $\Var{X}'$ is a smooth real (respectively, complex) embedded submanifold of $\bbk^n$ and its dimension as a real (respectively, complex) manifold is $\dim \Var{X}$.

For $x \in \calX'$, the linear subspace
\[
T_x \calX = \ker( \deriv{g}{x} )
\]
of $\bbk^n$ is the tangent space to $\calX$ at $x$. Note that $\dim T_x \calX  = \dim \calX$.

\subsection{Degree of a variety}\label{sec:degree}
Let $\mathcal X\subset \CC^n$ be a $d$-dimensional algebraic variety. A generic affine subspace of dimension $n-d$ intersects $\calX$ in finitely many points by Bertini's Theorem (see, e.g., \cite[Section~18]{Harris1992} or \cite[Theorem 0.5]{eisenbud-harris:16}). This number of points does not depend on the chosen subspace and is called the \emph{degree} of $\calX$.

The degree of a variety is a property of its embedding in $\CC^n$. An immediate consequence of Bezout's Theorem \cite[Theorem 7.7]{hartshorne2013algebraic} is the following, describing the degree of the image of a variety under the Veronese embedding $\nu_r$ from \cref{def-veronese}.
\begin{theorem}\label{thm:degVer}
Let $\mathcal X\subset \CC^n$ be an algebraic variety of dimension $d$ and degree~$\delta$. Then, $\nu_r(\mathcal X)$ is an algebraic variety of dimension $d$ and degree $r^{d}\cdot \delta$.
\end{theorem}
The only irreducible varieties of degree one are affine subspaces.

\subsection{Complexification of real algebraic varieties}\label{subsec: complexification}
Every real variety defines a complex variety via a process called complexification. Let $\Var{X}\subset \bbR^n$ be a real algebraic variety. The \emph{complexification} $\Var{X}^\mathbb{C}$ of $\Var{X}$ is
\[
\Var{X}^\mathbb{C}:=\{x\in \mathbb C^n  :  f(x) = 0 \text{ for all } f \in I^{\mathbb R}(\Var{X})\}.
\]
There are classic results relating the algebraic and geometric properties of a real algebraic variety to the ones of its complexification.

The following is a rephrasing of \cite[Lemmata 6 and 8]{whitney}:
\begin{theorem}\label{thm: withney lemma 6 and 8}
Let $\Var{X}\subseteq \mathbb R^n$ be a real algebraic variety and let $\Var{X}^{\mathbb C}$ be its complexification. Regard $\bbR^n$ as a (real) subspace of $\bbC^n$. Then,
\begin{enumerate}
 \item $\Var{X}^\bbC$ is the closure in the (complex) Zariski topology of $\Var{X}$;
 \item the set of real polynomials in $I^\bbC(\Var{X}^{\mathbb C})$ is $I^{\bbR}(\Var{X})$;
 \item the real dimension of $\Var{X}$ equals the complex dimension of $\Var{X}^\bbC$.
\end{enumerate}
\end{theorem}

\Cref{thm: withney lemma 6 and 8} implies that a generic point in a real algebraic variety $\Var{X}$ is also a generic point of the complexification $\Var{X}^{\bbC}$. More precisely, let $\Var{X}$ be a real variety, and let $U \subseteq \Var{X}^\bbC$ be the Zariski open subset on which a certain property $\calP$ holds. Then $U \cap \Var{X}$ is non-empty and Zariski open in $\Var{X}$.  In particular, Theorem \ref{thm: withney lemma 6 and 8} guarantees that results in complex algebraic geometry that hold for generic elements of a complex algebraic variety can be interpreted in the real setting.

\subsection{Maps between varieties}
For two varieties $\Var{X} \subseteq \bbk^\ell$ and $\Var{Y} \subseteq \bbk^n$, let
\begin{align*}
 \phi : \Var{X} \to \Var{Y}
\end{align*}
be a map and write $\phi(t) = (\phi_1(t) \vvirg \phi_n(t))$. We say that $\phi$ is a \emph{regular map} or a \emph{morphism of varieties} if, for every~$j=1 \vvirg n$, $\phi_j$ is (the restriction to $\Var{X}$ of) a polynomial in $\ell$ variables $t_1 \vvirg t_\ell$, which can be regarded as coordinates on $\bbk^\ell$.

The image of a morphism of varieties is not necessarily a variety. However, the following fundamental results characterize images of algebraic varieties in the complex and in the real case respectively.

\begin{theorem}[Chevalley's Theorem]\label{chevalley}
Let $\Var{X} \subseteq \bbC^\ell$ be a variety and let $\phi: \Var{X} \to \bbC^n$ be a morphism of varieties. Then $\Var{Y} = \phi(\Var{X})$ is locally closed in the Zariski topology. Hence, $\Var{Y}$ contains a Zariski open subset of its Zariski closure. Moreover, the closures of $\Var{Y}$ in the Zariski and the Euclidean topologies coincide.
\end{theorem}
\begin{proof}
 We refer to \cite[Theorem 4.19]{MichSturm:InvitationNonlinearAlgebra}.
\end{proof}

In the real setting, the analog result is the Tarski--Seidenberg Theorem.

\begin{theorem}[Tarski--Seidenberg Theorem]
Let $\Var{X} \subseteq \bbR^\ell$ be a semi-algebraic set and let $\phi : \bar{\Var{X}} \to \bbR^n$ be a morphism of varieties, where $\overline{\Var{X}}$ denotes the Zariski closure of $\Var{X}$ in $\bbR^\ell$. Then $\phi(\Var{X})$ is a semi-algebraic subset of $\bbR^n$.
\end{theorem}
\begin{proof}
 We refer to \cite[Section 2.2]{BCR1998}.
\end{proof}

In the complex case, the dimension of the (Zariski closure of the) image of an algebraic variety is controlled by the following result.
\begin{theorem}[Dimension of the fibers]\label{thm: theorem of dimension of the fibers}
 Let $\Var{X},\Var{Y}$ be complex irreducible varieties and let $\phi: \Var{X} \to \Var{Y}$ be a morphism such that $\phi(\Var{X})$ is (Zariski) dense in $\Var{Y}$. Then for every $y \in \phi(\Var{X})$,
 \[
\dim \phi^{-1}(y) \geq   \dim \Var{X} - \dim \Var{Y},
 \]
and there is a Zariski open subset of $\Var{Y}$ for which equality holds.
\end{theorem}
\begin{proof}
The proof relies on rank conditions of the differential of the map $\phi$. We refer to \cite[Theorem 11.22]{Harris1992} and \cite[Theorem 1.25]{Shafarevich2013} for a complete proof.
\end{proof}

Given a morphism of varieties $\phi: \bbk^\ell \to \bbk^n$, we say that $\phi$ is a \emph{parameterization} of $\im(\phi) \subseteq \bbk^n$. This is the setting (ES). In this case, for every $t \in \bbk^\ell$, one has a \emph{Jacobian map} $\deriv{\phi}{t} : \bbk^\ell \to \bbk^n$ defined by the differential of $\phi$ at $t$, that is
\begin{align*}
 \deriv{\phi}{t} : \bbk^\ell \to \bbk^n,\quad
 v \mapsto \left[ \frac{\partial}{\partial t_j} \phi_i (t) \right]_{\substack{i  = 1 \vvirg n\\ j = 1 \vvirg \ell}} \cdot v,
\end{align*}
where $\frac{\partial}{\partial t_j} \phi_i(t)$ denotes the partial derivative of the polynomial $\phi_i$ with respect to the variable $t_j$, evaluated at $t$. For $r \geq 0$, the set $V^\phi_r = \{ t \in \bbk^\ell  :  \rank ( \deriv{\phi}{t} ) \leq r\}$ is Zariski closed, and its defining equations are given by the $(r+1) \times (r+1)$ minors of the Jacobian matrix, regarded as polynomials in the coordinates $t_1 \vvirg t_\ell$ of~$\bbk^\ell$. Let $r_0 = \min\{r : V_r^\phi = \bbk^\ell\}$. Then $r_0$ is called the generic rank of $\phi$ and it is attained on the Zariski open subset $V_{r_0}^\phi \setminus V_{r_0-1}^\phi$.

In this setting, we provide the following theorem.
\begin{theorem}\label{thm: dimension image of a map}
Let $\phi:\bbk^\ell \to \bbk^n$ be an algebraic parameterization having generic rank $r_0$. Let $\Var{X}= \bar{\phi(\bbk^\ell)}$ be the closure in the Zariski topology of the image of $\phi$. Then $\Var{X}$ is an algebraic variety of dimension $r_0$.

Moreover, if $x = \phi(t)$ is a smooth point of $\calX$ with $t \in V_{r_0}^\phi$, then $T_x \calX = \im ( \deriv{\phi}{t} )$.
\end{theorem}
\begin{proof}
Restricting to smooth points of the image, for $\bbk=\RR$, the last statement follows from the Global Rank Theorem \cite[Theorem 4.14 a]{Lee2013}. As complex manifolds are in particular real manifolds the last statement also holds for $\bbk=\CC$. 

The statement about the dimension of $\Var{X}$ holds for a smooth point $x = \phi(t)$ of $\calX$ with $t \in V_{r_0}^\phi$ because
\[\dim\Var{X}=\dim T_x \calX =\dim \im ( \deriv{\phi}{t} )=r_0.\]
\end{proof}

Note that one can always choose generic points in the image of a parameterization like (ES). More precisely, if $\phi: \bbk^\ell \to \bbk^n$ is a polynomial parametrization of a subset of a variety $\Var{X} = \bar{\im (\phi)}$, and $\calP$ is a property that holds generically on $\Var{X}$, then the property $\calP$ also holds almost everywhere on the model $\im(\phi)$. This allows us to study the generic properties of the model $\im(\phi)$ through the closed algebraic variety $\Var{X}$.

A special role is played by the Veronese embedding $\nu_r : \bbk^n \to \bbk^{N(n,r)}$ where $N(n,r) = \binom{n+r}{r}$, which is
\begin{equation}
\begin{aligned}\label{def-veronese}
 \nu_r : \bbk^n &\to \bbk^{N(n,r)},\\
  \bfx & \mapsto (\bfx^\alpha:  \alpha \in \bbN^n, \vert \alpha \vert \leq r);
\end{aligned}
\end{equation}
that is, the parameterization defined by all the monomials of degree at most $r$ in the coordinates $\bfx= (x_1 \vvirg x_n)$ of $\bbk^n$. The properties of the Veronese embedding are summarized in the next immediate result.
\begin{lemma}\label{lemma: veronese properties}
 The map $\nu_r: \bbk^n \to \bbk^{N(n,r)}$ is an embedding. That is, $\nu_r$ is injective and its differential is injective at every point $x \in \bbk^n$. Moreover, if $f : \bbk^n \to \bbk$ is a polynomial of degree $r$, then there exists a linear map $L_f : \bbk^{N(n,r)} \to \bbk$ such that we have $f = L_f \circ \nu_r$.
\end{lemma}

Since polynomial maps are differentiable, a geometric version of the implicit function theorem holds.

\begin{theorem}\label{thm: implicit function theorem}
 Let $\Var{X},\Var{Y}$ be irreducible varieties, with $\dim \calX = d$ and $\dim \calY = m$. Let $\phi: \Var{X} \to \Var{Y}$ be a morphism of varieties with $\calY = \bar{\phi(\calX)}$. Then, there exist Zariski open subsets $U_\calX \subseteq \calX$ and $U_\calY \subseteq \calY$ such that $U_\calX$ and $U_\calY$ are smooth manifolds, $\phi|_{U_\calX} : U_\calX \to U_\calY$ is a smooth map and for every $y \in U_\calY$, the fiber $\phi^{-1}(y)$ is a smooth manifold of dimension $d-m$.
\end{theorem}
\begin{proof}
Let $U'_\calX \subset \calX$ be the subset of smooth points on which $\phi$ is regular; that is,
\[
U'_\calX = \{ x \in \calX : x \text{ is smooth in $\calX$ and } \rank( \deriv{\phi}{x} ) = m\}.
\]
By Sard's Theorem \cite[Theorem 6.8]{Lee2013}, $U'_\calX$ is non-empty and it is Zariski open because it is the intersection of two Zariski open subsets. Let $U_\calY$ be the subset of $\phi(U'_\calX)$ consisting of smooth points of $\calY$. Then $U_\calY$ is Zariski open in $\calY$. Let $U_\calX = \phi^{-1}(U_\calY)$. Since $\phi$ is a regular map, $U_\calX$ is Zariski open in $\calX$.

The restriction $\phi|_{U_\calX} : U_\calX \to U_\calY$ is a smooth map of manifolds, and for every $x \in U_\calX$, $\rank(\deriv{\phi}{x} ) = m$. By \cite[Theorem 5.22]{Lee2013}, we conclude that the fibers of $\phi|_{U_\calX}$ are smooth manifolds of dimension $d-m$.
\end{proof}

A consequence of \cref{thm: implicit function theorem} is that the rank of the differential at a point controls the local injectivity of the map. We have the following result.

\begin{corollary}\label{corol: injective differential implies injective map}
Let $\calX \subseteq \bbk^n$ be an irreducible variety. Let $\hat{\mu} : \bbk^n \to \bbk^s$ be a morphism of varieties and let $x \in \calX$. If $\ker (\deriv{\hat{\mu}}{x}) \cap T_x \calX = 0$, then there exists a Euclidean open neighborhood $U_x \subseteq \calX$ of $x$ such that $\hat{\mu}|_{U_x} : U_x \to \bbk^s$ is a diffeomorphism onto its image.
\end{corollary}
\begin{proof}
Let $\psi : \calX \to \bbk^s$ be the restriction of $\hat{\mu}$ to $\calX$. Then the kernel of its derivative is $\ker (\deriv{\psi}{x} ) = \ker ( \deriv{\hat{\mu}}{x} ) \cap T_x \calX = 0$, so that $\rank(\deriv{\psi}{x}) = \dim \calX$. 
This implies $x$ is a smooth point of $\calX$. Hence, an open neighborhood of $x$ is a $\bbk$-manifold. Since $\rank(\deriv{\psi}{x}) = \dim \calX$, the assumptions of the inverse function theorem for manifolds are satisfied. For $\bbk=\bbR$, the claim follows from \cite[Theorem 4.5]{Lee2013}.

For $\bbk=\bbC$, the proof of \cite[Theorem 4.5]{Lee2013} applies by replacing ``smooth manifold'' with ``complex manifold'' \cite[Chapter IV]{FG2002}, ``smooth map'' with ``holomorphic function'' \cite[Chapter I]{FG2002}, ``diffeomorphism'' with ``biholomorphism'' \cite[Chapter I]{FG2002}, and using the inverse function theorem for holomorphic functions \cite[Chapter I, Theorem 7.5]{FG2002} instead of the one for reals.
\end{proof}

Finally, we need the theorem of birationality to a hypersurface. It is a strong result on generic identifiability that does not depend on the geometry of the variety.
\begin{theorem}[Birationality to a hypersurface]\label{thm:bir with hyper}
Let $\Var{X}\subseteq \CC^n$ be a $d$-dimensional variety. Let $\hat{\mu} : \bbC^n \to \bbC^{d+1}$ be a generic projection. Then $\mu=\hat{\mu}|_\Var{X}$ is generically one-to-one, i.e., a generic point of the image of $\mu$ has a unique preimage in $\Var{X}$.
\end{theorem} 
\begin{proof}
We refer to \cite[Theorem 1.8 and Remark 1.2]{Shafarevich2013}.
\end{proof}

\subsection{The Noether Normalization Lemma}\label{sec:complex_AG}
The main results about the required number of measurements $s$ for recoverability are consequences of the following fundamental theorem in commutative algebra.

\begin{theorem}[Noether's Normalization Lemma]\label{thm:Noether Norm}
Let $\bbk$ be an infinite field and $\Var{X}\subset \bbk^n$ be a variety of dimension $d$ and degree $\delta$. Let $\hat{\mu}: \bbk^n \to \bbk^d$ be a generic affine linear projection and let $\mu=\hat{\mu}|_\Var{X}$ be its restriction to $\Var{X}$. Then $\mu : \Var{X} \to \bbk^d$ is a finite map. In particular, every fiber of $\mu$ is finite.
\end{theorem}
\begin{proof}
The proof of this theorem is classic. We provide a sketch of the proof following \cite[Chapter 10]{Gathmann:CommAlgebra}.

The usual algebraic formulation of Noether's normalization lemma is as follows.
Let $R$ be a finitely generated algebra over $\bbk$ with generators $x_1 \vvirg x_n$. Then, there exists an integer $d$ and algebraically independent elements $r_1 \vvirg r_d \in R$ such that $R$ is a finite module over $\bbk[r_1 \vvirg r_d]$. Moreover, $r_1 \vvirg r_d$ can be chosen to be linear combinations of the generators $x_1 \vvirg x_n$.
This statement is proved in \cite[Proposition 10.5]{Gathmann:CommAlgebra}.

The geometric statement is obtained by taking the affine coordinate ring of $\Var{X}$, i.e., $R = \bbk[\Var{X}] = \bbk [\bbk^n] / I(\Var{X})$, which is finitely generated with generators $x_1 \vvirg x_n$, the (equivalence classes modulo $I(\Var{X})$ of the) coordinates of $\bbk^n$. The algebraic statement guarantees that there exist linear combinations $r_1 \vvirg r_d$ of the $x_j$'s such that $\bbk[\Var{X}]$ is a finitely generated module over $\bbk[r_1 \vvirg r_d]$. In fact, the proof \cite[Proposition 10.5]{Gathmann:CommAlgebra} and \cite[Lemma 10.3]{Gathmann:CommAlgebra} show that a generic choice of linear combinations satisfies the statement. The inclusion map
\(
 \bbk[r_1 \vvirg r_d] \to R
\)
defines an algebraic map $\mu : \Var{X} \to \bbk^d$ as $\bbk[r_1 \vvirg r_d]$ can be regarded as the coordinate ring of the affine space $\bbk^d$. The fact that the $r_j$'s are linear combinations of the $x_i$'s guarantees that $\mu$ is the restriction of a linear map $\hat{\mu} : \bbk^n \to \bbk^d$.
\end{proof}
 
Noether's Normalization Lemma is sharp in the following sense:
\begin{enumerate}
\item For any affine linear projection $\hat{\mu}: \bbk^n \to \bbk^{d-1}$, the fiber of $\mu=\hat{\mu}|_\Var{X}$ over $y\in\bbk^{d-1}$ may be infinite. Further, if $\bbk$ is algebraically closed, then the generic fiber of $\mu$ is infinite.
\item It is possible that special projections $\hat{\mu}: \bbk^n \to \bbk^d$ map $\Var{X}$ to a proper subvariety of $\bbk^d$. In this case, the fiber of $\mu=\hat{\mu}|_\Var{X}$ at a generic point of $\bbk^d$ is empty and any fiber over the image of $\mu$ is infinite.
\end{enumerate}

An example when case $(2)$ holds is a projection of any product of (positive dimensional) varieties $\calX_1\times \calX_2\subset \bbk^{n_1}\times\bbk^{n_2}=\bbk^{n_1+n_2}$ to $\bbk^{n_1}$.

\begin{remark}
The finiteness condition of the map $\mu : \Var{X} \to \bbk^d$ in the statement of the Noether normalization lemma has stronger implications than the fact that the fibers are finite. In fact, \cref{thm:Noether Norm} implies that for every coordinate $x_j$ of $\bbk^n$, there exists a fixed univariate polynomial $P_j \in \bbk[y_1,\dots,y_d][x_j]$, (i.e., the coefficients of $P_j$ are polynomials on~$\bbk^d$ depending on $\Var{X}$ and on $\mu$) with the property that for every $y\in \bbk^d$, the $j$th coordinates of the points of $\mu^{-1}(y)$ are the roots of $P_j(y,-) \in \bbk[x_j]$. Once we find the polynomials $P_j$, the $j$th coordinates of the points in the fiber of $\mu$ thus vary like the roots of $P_j$ as $y \in \bbk^d$ varies.
This provides an additional viewpoint on the continuity of the solutions of the inverse problem \cref{eqn_ip}, which we investigated in \cref{sec_continuity}.
\end{remark}


\section{Proofs for the Q\&A section}\label{q_and_a_proofs}
In this section, we give complete proofs of the correctness of the answers provided in \cref{q_and_a}. The proofs in this section exploit the fact that a generic point in $\Var{X}$ is a generic point in the complexification $\Var{X}^{\mathbb C}$ of $\Var{X}$, and that $d=\dim \Var{X}$ equals the complex dimension of $\Var{X}^{\mathbb C}$ (this is true irrespective of $\Var{X}$ being real or complex). For models given in implicit form (IS) this is implied by \cref{thm: withney lemma 6 and 8}. For models in explicit form (ES) we can use \cref{thm: dimension image of a map}.

\subsection{The answer to \cref{QExistence}}
The restatement in terms of a system of polynomial equations is straightforward. The fact that any system of polynomial equations over $\CC$ has a solution if and only if the ideal generated by the polynomials does not contain $1$ is one of the forms of Hilbert's Nullstellensatz \cite[Corollary 1.6]{Eis:CommutativeAlgebra}. Similarly, the statement over the reals can be found in \cite[Corollary 4.1.8]{BCR1998}. The fact that checking if an ideal contains $1$ may be done using Gr\"obner bases follows from the definition of Gr\"obner basis. \qed

\subsection{The answer to \cref{question:finite map}}
If $s \geq d$, \cref{thm:Noether Norm} guarantees that a generic projection has finite fibers. The inequality is sharp, both in the real and in the complex case. Indeed, suppose $s \leq d-1$ and distinguish two cases.

First, if $\bbk=\CC$, by \cref{thm: dimension image of a map}, there exists a Zariski open subset $U\subset \mu(\Var{X})$ such that for any $y\in U$ we have $\dim  \mu^{-1}(\mu(y)) +\dim \mu(\Var{X})=\dim \Var{X}$. In particular, since $\dim \mu(\Var{X})\leq s < d$, we have $\dim \mu^{-1}(\mu(x))>0$. This implies that the generic fibers are infinite.

Second, if $\bbk=\RR$, the map $\mu: \Var{X}\rightarrow \RR^s$ is a morphism of varieties. Therefore, by \cref{thm: implicit function theorem}, its generic fibers contain manifolds of dimension equal to $\dim \calX - \dim (\im (\mu)) \geq \dim \calX - s$. In particular, if $s < d$, generic fibers are infinite. \qed

\subsection{The answer to \cref{question:finite map2}}
By \cref{lemma: veronese properties} and Theorem \ref{thm:degVer}, this reduces to \cref{question:finite map}. \qed

\subsection{The answer to \cref{Q_coordinate_recovery}}
The condition $s\geq d$ is necessary as in the answer to \cref{question:finite map}. We prove that a coordinate projection exists for $s\geq d$.

Assume that $s =d$.
Let $x\in \Var{X}$ be a smooth point. Let $v_1,\dots, v_d$ be a spanning set for $T_x\calX$, the tangent space to $\calX$ at $x$. Let $M$ be the matrix whose rows are the vectors $v_i$. Hence, $M$ is a $d \times n$ matrix with $\rank(M) = d$. In particular, $M$ contains a $d\times d$ submatrix $N$ with $\rank(N) = d$. This submatrix is obtained by selecting $d$ columns of $M$, or, equivalently, $d$ coordinates in $\bbk^n$. We claim that the projection~$\hat{\mu}$ to these~$d$ coordinates is generically recoverable when restricted to $\Var{X}$.

Since $\hat{\mu}$ is linear, the differential $\deriv{\hat{\mu}}{x}$ coincides with $\hat{\mu}$ itself. Now, its image $\im ( \deriv{\hat{\mu}}{x}  : \bbk^n \to \bbk^d)$ restricted to $T_x \calX$ is the column span of the matrix $N$, which coincides with the whole $\bbk^d$. Thus, $\deriv{\hat{\mu}}{x}$ is of full rank and $T_x \calX \cap \ker( \deriv{\hat{\mu}}{x}) = 0$. By \cref{corol: injective differential implies injective map}, there is a Euclidean open neighborhood $\tilde{U}$ of $x\in\Var{X}$ such that $\mu=\hat{\mu}|_\calX$ is a diffeomorphism of $\tilde{U}$ onto its image.

We distinguish two subvarieties $\Var{S}_1,\Var{S}_2$ of $\calX$. Let $\Var{S}_1$ be the variety of singular points of $\calX$, which is a proper subvariety of $\calX$. Let $\Var{S}_2$ be the set consisting of the points $x \in \calX$ such that $T_x \calX \cap \ker(\deriv{\hat{\mu}}{x}) \neq 0$. This is a closed condition. Since $\tilde{U} \subseteq \calX \setminus \Var{S}_2$, we deduce that $\Var{S}_2$ is a proper subvariety of $\calX$. Let $U = \calX \setminus (\Var{S}_1 \cup \Var{S}_2)$, which is Zariski open in $\calX$.

We can conclude that the fiber over every point $y\in \bbk^d$ with $y \in \mu(U)$ is finite. Indeed, suppose that it has positive dimension. Then there exists an $x \in \mu^{-1}(y)$ that is smooth, $T_x \calX \cap \ker(\deriv{\hat{\mu}}{x}) = 0$, and every Euclidean neighborhood of $x$ contains infinitely many points of $\mu^{-1}(y)$. But this contradicts \cref{corol: injective differential implies injective map}. \qed
\begin{remark}
We note that a central point of the previous proof is a simple fact from linear algebra, that if we have a full rank $r$ rectangular matrix, then we can choose a square submatrix of rank $r$. For a matroidal perspective on this proof we refer to \cite{rosen2020algebraic, Tsakiris2021b}.  
\end{remark}

\subsection{The answer to \cref{question: specific map is recoverable}}

In the implicit setting (IS) our model is given by  $\calX = \{ x \in \bbk^n : f_1 (x) = \cdots = f_k(x) = 0\}$. Let $x \in \calX$ be a point satisfying the condition described in the answer. We note that this is an open condition, i.e.,~if it holds for some $x$, then there is a Zariski open set $U\subset\Var{X}$ such that the condition holds for all $x\in U$.

Let us consider the Zariski open set $\mathcal X^{\mathrm{sm}}\subset\Var{X}$ of smooth points of $X$. Taking $x\in U\cap \mathcal X^{\mathrm{sm}}$ we see that
$\ker ( \deriv{\mu}{x} ) \cap \ker (\deriv{f}{x}) = 0$ implies that the restriction of $\deriv{\mu}{x}$ to~$T_x\calX$ is injective.
By \cref{corol: injective differential implies injective map}, we deduce that $\mu$ is injective in a Euclidean neighborhood of $x$. The same argument as in the answer to \cref{Q_coordinate_recovery} shows that $\mu$ is generically recoverable.

In the explicit setting (ES), let $\calX = \bar{\im(\phi)}$ where $\phi : \bbk^\ell \to \bbk^n$ is a polynomial parameterization. The map $\mu \circ \phi$ parameterizes the image $\mu( \calX)$. If $t^* \in \bbk^\ell$ is a point with $\rank ( \deriv {( \mu \circ \phi)}{t^*} ) = d$, then $\rank \bigl( \deriv{(\mu \circ \phi)}{t} \bigr) \geq d$ for generic $t \in \bbk^\ell$ by semicontinuity.
By \cref{thm: dimension image of a map}, we obtain $\dim \mu( \calX) \geq d$. Since $\dim\calX = d$, $\mu : \calX \to \bbk^s$ maps $\calX$ onto a (possibly semi-algebraic) set of the same dimension. By \cref{thm: implicit function theorem}, the generic fiber of~$\mu$ is $0$-dimensional, so that~$\mu$ is generically recoverable. \qed

\subsection{The answer to \cref{Q_which_coordinate_projection}}
The answer to \cref{question: specific map is recoverable} entails that the coordinate projection should satisfy the condition that the differential of $\mu \circ \phi$ has rank $d$ at some point $t \in \bbk^\ell$. The matrix
\(
\left[\begin{smallmatrix}  \frac{\partial \phi_{i_p}}{\partial t_j} \end{smallmatrix}\right]_{\substack{p = 1 \vvirg s \\ j = 1 \vvirg \ell}}
\)
is precisely the matrix representation of this differential.\qed

\subsection{The answer to \cref{Q3}}
For any field, if $\Var{X}$ is an affine space then a generic projection to $\bbk^d$ is identifiable everywhere.

Consider the case  $\bbk=\CC$.
For $s\geq d+1$ the answer follows from \cref{thm:bir with hyper}.
Assume now that a generic projection $\hat{\mu}$ from $\CC^n$ to $\CC^d$ is generically identifiable. Let $\mu=\hat{\mu}|_\calX$. If $\dim \mu(\Var{X})<d$ then, as in the proof for \cref{question:finite map}, the generic fibers would be of positive dimension. Hence, we may assume that $\mu(\Var{X})$ is dense in $\CC^d$ by \cref{chevalley}.
The kernel of $\hat{\mu}$ is a generic linear subspace in $\CC^n\supset \Var{X}$. Hence, the fiber of $\hat{\mu}$ over a generic point of $\mu(\Var{X})$ is a generic affine subspace of~$\CC^n$ of codimension $d$. Thus, it must intersect $\Var{X}$ in $\delta$ distinct points, where $\delta$ is the degree of $\Var{X}$. Since $\mu$ is generically identifiable, this means that $\Var{X}$ must have degree one, i.e., it is an affine subspace (see the discussion at the end of \cref{sec:degree}).

For $\bbk=\RR$ and $s\geq d+1$ we may take the complexification of the variety $\Var{X}$ and see that a generic complex, hence, a fortiori real, fiber is unique.
To see that $s=d$ measurements do not suffice, we may take $\Var{X}$ to be a sphere.
Generic points in the image $\mu(\calX)$ of a generic projection of the sphere have $2$ elements in the preimage.

If $s<d$, as in the proof for \cref{question:finite map}, generic fibers have positive dimension. \qed

\subsection{The answer to \cref{Q4}}
As in the proof for the answer to \cref{question:finite map2} we conclude by applying \cref{lemma: veronese properties}.

We note that for $d>0$ and $p>1$ the Veronese map strictly increases the degree, thus we may exclude the case of the affine subspace in \cref{Q3}.
\qed

\subsection{The answer to \cref{ques:inj}}
Let $\hat{\mu}:\bbk^n\to\bbk^s$ be the linear map that restricts to $\mu$ on $\Var{X}$.
For distinct $x_1,x_2\in \Var{X}$ we have that $\hat{\mu}(x_1)=\hat{\mu}(x_2)$ is equivalent to
\begin{align*}
\hat{\mu}(x_1-x_2)=0
\quad\Leftrightarrow\quad \exists p \in \Delta(\Var{X})\setminus\{0\} : \hat{\mu}(p)=0
\quad\Leftrightarrow\quad \ker(\hat{\mu}) \cap \Delta(\Var{X})\neq 0.
\end{align*}
Thus, the identifiability property depends only on the kernel of $\hat{\mu}$. The kernel of a generic $\hat{\mu}$ is a generic linear subspace of codimension $s$. A generic linear subspace of codimension $s$ intersects $\Delta(\Var{X})$ only in zero if $s> \dim \Delta(\Var{X})$ (cf.~\cref{sec:degree}), which proves the first equality in the answer.
As $\Delta(\Var{X})$ may be realized as the image of $\Var{X}\times \Var{X}$ we must necessarily have the inequality $\dim \Delta(\Var{X})\leq 2d$.\qed

\subsection{The answer to \cref{ques:inj1}}
Since Zariski open subsets are also Euclidean open, the answer to \cref{ques:inj} concludes the proof.
\qed

\subsection{The answer to \cref{Q8}}
If $s<d$, then the answer to \cref{question:finite map} shows that $\mu$ cannot be a local diffeomorphism, because it is not recoverable generically.

Let $s\geq d$. We denote the set of singular points in $\mu(\Var{X})$ by $\Var{S}$, so that $\mu(\Var{X})\setminus \Var{S}$ is a smooth embedded submanifold of $\bbk^s$ of dimension equal to the dimension of~$\mu(\Var{X})$.

The preimage $\mu^{-1}(\Var{S}) \cap \mathcal X$ is a proper subvariety of $\mathcal X$. As in the answer to \cref{Q_coordinate_recovery}, we consider two additional subvarieties of~$\Var{X}$. The first~$\Var{S}_1$ is the set of singular points, which is a proper subvariety.
The second one~$\Var{S}_2$ is the set of points $x\in \Var{X}$ such that $(T_x\Var{X})\cap \ker(\deriv{\hat{\mu}}{x}) \neq 0$. Since $\dim \ker(\hat{\mu}) = n - s \leq n - d$, the variety $\Var{S}_2$ is a proper subvariety. Then, we define
\(
\calX' := \calX \setminus (\mu^{-1}(\Var{S}) \cup \Var{S}_1 \cup \Var{S}_2).
\)
This is a Zariski open subset of $\mathcal X$ contained in the smooth locus of $\mathcal X$. Hence,~$\mathcal X'$ is a smooth embedded submanifold of~$\bbk^s$. By construction, $\mu(\Var{S}_1)\cup \mu(\Var{S}_2)$ is of lower dimension in $\mu(\calX)$. Therefore, it is also of lower dimension in the manifold $\mu(\calX)\setminus \Var{S}$. This implies that $\mu(\calX') $ is an open subset of a smooth embedded submanifold of~$\bbk^s$, hence it is itself a smooth embedded submanifold of $\bbk^s$.

The derivative
$\deriv{(\mu|_{\Var{X}'})}{x}$ is an isomorphism of the tangent spaces $T_x \Var{X}'$ and $T_{\mu(x)}\mu(\Var{X}')$. By the inverse function theorem for manifolds, $\mu|_\Var{X'}$ is a local diffeomorphism onto its image.
\qed

\begin{remark}
Note, however, in general, for $s=d$ there will be no Zariski open neighbourhoods that would give isomorphic algebraic varieties. This means that the inverse map will be analytic but not algebraic.
\end{remark}
\subsection{The answer to \cref{Q9}}
For $s \ge d+1$, the answer to \cref{Q3} implies that $\mu$ is identifiable for generic $x\in \Var{X}$. Let $\Var{S},\Var{S}_1,\Var{S}_2$ be as in the answer to \cref{Q8} and let $\Var{S}_3\subset \calX$ be subvariety of points where $\mu$ is not identifiable. Let
\(\calX' := \calX \setminus ( \mu^{-1}(\Var{S}) \cup \Var{S}_1\cup \Var{S}_2\cup \Var{S}_3).\)
As in the answer to \cref{Q8} we have that $\mathcal X'$ and $\mu(\mathcal X')$ are smooth embedded submanifolds of $\bbk^s$ and that $\mu|_{\calX'}$ is a local diffeomorphism. Since $\mu|_{\calX'}$ is also injective, it is a global diffeomorphism.

It remains to prove that if $\bbk=\CC$, $s=d$ and $\Var{X}$ is not an affine space then we do not have a global diffeomorphism. This follows from the answer to \cref{Q3}.
 \qed

\subsection{The answer to \cref{Q11}}
Recall the definitions of the singular locus $\Var{S}_1\subset \mathcal X$ and the set of points $\Var{S}_2\subset \mathcal X$ for which $(T_x\Var{X})\cap \ker(\deriv{\hat{\mu}}{x}) \neq 0$ from the answer to \cref{Q8}. The smooth locus is $\mathcal X^{\mathrm{sm}} := \mathcal X\setminus \mathcal S_1$. We also set $\mathcal X':=\mathcal X^\mathrm{sm}\setminus \mathcal S_2$.
As in the answer to \cref{Q8}, both $\mathcal X^\mathrm{sm}$ and $\mathcal X'$ are Zariski open submanifolds of $\mathcal X$. We show the asserted characterization of the condition number for smooth points $x\in \mathcal X^{\mathrm{sm}}$. We distinguish two cases.

The first case is when $x\in\calX'$. In this case we have $(T_x\Var{X})\cap \ker(\deriv{\hat{\mu}}{x}) = 0$, which implies that the derivative $\deriv{\mu}{x}$
is of full rank. This means that $\mu|_{\Var{X}'}: \Var X' \to \bbk^s$
is an immersion. By \cite[Theorem 4.25]{Lee2013} (for $\bbk=\bbC$ the same argument works), there exists an open neighbourhood $V_x\subset \Var X'$ of $x$ such that 
$\mu(V_x)$ is an embedded submanifold of $\bbk^s$ \cite[Section 5]{Lee2013}.
Hence, there exists a local inverse map $\mu_x^{-1}: \mu(V_x)\to V_x$. Let $y=\mu(x)$. The condition number at $x$ is given by \cref{eqn_kappa}:
\[
\kappa(x) = \ \lim_{\epsilon\to 0}\ \sup_{\substack{y' \in \mu(V_x), \Vert y - y'\Vert \le \epsilon}}\ \frac{\Vert \mu_x^{-1}(y) - \mu_x^{-1}(y') \Vert}{\Vert y - y' \Vert},
\]
where $\Vert\cdot \Vert$ is the Euclidean norm on both $\bbk^n$ and $\bbk^s$. Since $\mu_x^{-1}$ is smooth (holomorphic for $\bbk=\bbC$) and $\mu(V_x)$ is an embedded manifold, we can use \cite[Theorems 3 and 4]{Rice1966} to obtain
\[
 \kappa(x) = \ \max_{u\in T_y \mu(V_x), \|u\|=1} \Vert (\deriv{\mu_x^{-1}}{y}) \, u\Vert = \sigma_1( \deriv{\mu_x^{-1}}{y} ),
\]
where $\sigma_1$ denotes the largest singular value.
Because of the inverse function theorem \cite[Theorem~4.5]{Lee2013} (or its complex analogue), $\deriv{\mu_x^{-1}}{y} = (\deriv{\mu}{x})^{-1}$. Hence, we have $\kappa(x) = \frac{1}{\sigma_d( \deriv{\mu}{x} )}$ because $\dim V_x = \dim \mu(V_x) = d$ by the answer to \cref{Q8}.
The reformulation in terms of matrices is immediate.

The second case is when $x\in \mathcal X^{\mathrm{sm}} \setminus \mathcal X'$ so that $(T_x\Var{X})\cap \ker(\deriv{\mu}{x}) \neq 0$. Then, $\deriv{\mu}{x}$ has a nontrivial kernel, which implies that there is a smooth (holomorphic for $\bbk=\bbC$) curve $\gamma(t) \subset \calX^{\mathrm{sm}}$ with $\gamma(0)=x$ and $\tfrac{\mathrm{d}}{\mathrm{d} t} \mu(\gamma(t)) |_{t=0} = 0$.
We can bound the condition number as $\kappa(x) \geq \ \lim_{t\to 0}\ \frac{\Vert x - \gamma(t) \Vert}{\Vert y - \mu(\gamma(t))\Vert} = \infty$.
On the other hand,
$(T_x\Var{X})\cap\ker \deriv{\mu}{x}\neq 0$ implies $\sigma_{d}(\deriv{\mu}{x}) = 0$. Therefore, $\kappa(x) = \tfrac{1}{\sigma_{d}(\deriv{\mu}{x})}$ also in this case. \qed


\section{A numerical experiment}\label{sec_numexp}
We confirm \cref{demo-thm} experimentally in a few examples.\footnote{The Julia code we used is attached as an ancillary file and available at \cite{tensor-recovery2021}. The experiments were performed in Julia v1.6.2 on a computer running Ubuntu 20.04.2 LTS consisting of an Intel Core i7-4770K CPU (4 cores, 3.5GHz clockspeed) with $32$GB of main memory.} We showcase numerical homotopy continuation methods from HomotopyContinuation.jl \cite{BT2018} to accurately solve small algebraic compressed sensing problem instances. Recall that homotopy methods are numerical techniques for computing all isolated solutions of a system of polynomial equations. Additional experiments confirming \cref{demo-thm} with alternative techniques were independently conducted by Kr\"amer \cite{Kraemer2021}.

The structure we consider is that of tensors with a low-rank tensor rank decomposition \cite{Hitchcock1}. Specifically, let $\Var{T}_2$ be the set of tensors in $\bbR^{4}\otimes\bbR^{3}\otimes \bbR^{2} \simeq \bbR^{4 \times 3 \times 2}$ of rank bounded by $2$. This set $\Var{T}_2$ can be parameterized explicitly as in (ES) with
\begin{align*}
\phi'(a_1',a_2',b_1',b_2',c_1',c_2')
&= a_1'\otimes b_1'\otimes c_1' + a_2'\otimes b_2'\otimes c_2'\\
&= [ (a_1')_i(b_1')_j(c_1')_k ]_{i,j,k=1}^{4,3,2} + [ (a_2')_i(b_2')_j(c_2')_k]_{i,j,k=1}^{4,3,2},
\end{align*}
where $a_1',a_2'\in\bbR^4$, $b_1',b_2'\in\bbR^{3}$, and $c_1',c_2'\in\bbR^{2}$.

We want to compute isolated solutions of the polynomial system in \cref{demo-thm} with numerical homotopy continuation methods to find alternative decompositions into rank-$1$ terms, depending on the measurement map $\mu$. This requires a parameterization $\phi$ from a Euclidean space of dimension $d = \dim \Var{T}_2$ into $\Var{T}_2$.
However, it is known that this is not always possible. 
In general $\Var{T}_2$ needs to be covered with multiple charts $\bbR^{d} \to \Var{T}_2$.

The dimension of $\Var{T}_2$ is known to be $d=2(4+3+2-2)=14$ \cite{AOP2009} but the domain of $\phi'$ has dimension $2(4+3+2)=18$, so clearly $\phi'$ overparameterizes~$\Var{T}_2$.
We restrict ourselves to parameterizing the Zariski open subset $\Var{X}$ of $\Var{T}_2$ given by the image of
\[
 \phi(a_1,a_2,b_1,b_2,c_1,c_2) = \begin{bmatrix}a_1\\ 1\end{bmatrix}\otimes \begin{bmatrix}b_1\\1\end{bmatrix}\otimes c_1 + \begin{bmatrix}a_2\\1\end{bmatrix}\otimes \begin{bmatrix}b_2\\1\end{bmatrix}\otimes c_2,
\]
where now $a_1,a_2\in\bbR^{3}$, $b_1,b_2\in\bbR^2$, and $c_1,c_2 \in \bbR^2$. This $\Var{X}$ is the Zariski open subset of tensors of rank bounded by $2$ that have at least one decomposition where all vectors in the first and second factors $\bbR^4$ and $\bbR^3$ are nonzero in the last element.

Given a measurement map $\hat{\mu} : \bbR^{4\times3\times2} \to \bbR^s$ and a measurement $y = \mu(T)$ of a rank-$2$ tensor $T \in \Var{X}$, we can construct the polynomial system from \cref{demo-thm},
\begin{align}\label{eqn_poly_sys}\tag{PS}
y - \mu(\phi(a_1,a_2,b_1,b_2,c_1,c_2)) = 0,
\end{align}
and solve for the variables $a_1,a_2,b_1,b_2,c_1,c_2$.

We investigate the interesting case of random coordinate projections, as the behavior for generic linear projections is completely described in the Q\&A section. The minimum number of measurements for generic recoverability is $s=d$ by \cref{demo-thm}. So let us start with this many measurements.
We consider the following incomplete $4\times 3\times 2$ tensor of rank 2 (the fact that it has rank at most $2$ will be implied by \cref{demo-thm} after we have solved the system \cref{eqn_poly_sys} and found a solution):
\[
 y = \left[\begin{array}{ccc|ccc}
  -32   & ?     & -24   & ?   & 10    & ? \\
  72    & ?     & ?     & -57   & 27    & ? \\
  -104  & 40    & ?     & -11   & 1     & ? \\
  ?     & 16    & 0     & -1    & ?     & -7
 \end{array}\right],
\]
where the left submatrix is the first $4\times3$ slice of $y \in \bbR^{4\times3\times2}$ and the right submatrix is the second slice. Solving the system \cref{eqn_poly_sys}, we find $4$ isolated, nonsingular, real solutions $a_1^i, a_2^i, b_1^i, b_2^i, c_1^i, c_2^i$, $i=1,\ldots,4$. One can show that the parameterization $\phi$ is generically $2$-to-$1$: swapping the variables $a_1 \leftrightarrow a_2$, $b_1 \leftrightarrow b_2$, and $c_1 \leftrightarrow c_2$ corresponds to summing the rank-$1$ terms in a different order, which produces the same result. Consequently, only $2$ of these $4$ solutions represent distinct points of $\Var{X} \subset \Var{T}_2$. These distinct completed tensors, with imputed entries highlighted and rounded to $3$ decimals, are
\begin{align*}
 T_1 &= \left[\begin{array}{ccc|ccc}
  -32   & \mathbf{8}     & -24   & \mathbf{-26}   & 10    & \mathbf{-2} \\
  72    & \mathbf{-40}   & \mathbf{-5}    & -57   & 27    & \mathbf{21} \\
  -104  & 40    & \mathbf{-8}    & -11   & 1     & \mathbf{-17} \\
  \mathbf{-40}   & 16    & 0     & -1    & \mathbf{-1}    & -7
 \end{array}\right], \text{ and }\\
 T_2 &=
 \left[\begin{array}{ccc|ccc}
  -32       & \mathbf{411.518}   & -24           & \mathbf{-26}   & 10        & \mathbf{-182.054} \\
  72        & \mathbf{1112.908}  & \mathbf{1075.769}      & -57   & 27        & \mathbf{-396.574} \\
  -104      & 40        & \mathbf{-728.216}      & -11   & 1         & \mathbf{-78.642} \\
  \mathbf{-1.114}    & 16        & 0             & -1    & \mathbf{0.389}     & -7
 \end{array}\right],
\end{align*}
and their (unique) decompositions into rank-$1$ terms, rounded to $3$ decimals, are
\begin{align*}
 T_1 &= \begin{bmatrix} -1\\-6\\2\\1 \end{bmatrix} \otimes \begin{bmatrix} -2\\1\\1 \end{bmatrix} \otimes \begin{bmatrix}8\\-4\end{bmatrix} + \begin{bmatrix}2\\1\\3\\1\end{bmatrix}\otimes\begin{bmatrix}3\\-1\\1\end{bmatrix}\otimes\begin{bmatrix}-8\\-3\end{bmatrix}, \text{ and }\\
 T_2 &= \begin{bmatrix} 25.925\\60.344\\8.736\\1 \end{bmatrix} \otimes \begin{bmatrix} 0.139\\-0.001\\1 \end{bmatrix} \otimes \begin{bmatrix}-116.229\\-9.791\end{bmatrix} +
 \begin{bmatrix}25.719\\69.6\\2.471\\1\end{bmatrix}\otimes\begin{bmatrix}0.129\\0.137\\1\end{bmatrix}\otimes\begin{bmatrix}116.229\\2.791\end{bmatrix}.
\end{align*}

We verified that the differential of $\mu \circ \phi$ at all $4$ solutions for the variables $a_1,a_2,b_1,b_2,c_1,c_2$ was of (numerical) rank equal to $d=14=\dim \Var{T}_2$. Hence, these solutions are locally continuous by the inverse function theorem. We computed the (absolute) condition number of the inverse problem \cref{eqn_ip} at the solution $T_1$, respectively $T_2$, and found approximately $11.642$, respectively $1907.954$.

It is interesting to observe that both $T_1$ and $T_2$ coincide, up to $14$ significant digits, in the unobserved coordinate $(1,1,2)$ with value $-26$. This implies that augmenting $\mu$ with a projection to this coordinate (so it projects to $s = d+1$ coordinates) will not eliminate either completed tensor, even though this number of measurements generically suffices for generic identifiability by \cref{demo-thm}!
On the other hand, additionally projecting to one of the other coordinates eliminates either $T_1$ or $T_2$, so a unique solution remains with $s = d + 1$ even for such a coordinate projection (which are not generic linear maps in the sense of \cref{demo-thm}). 

\section{Conclusions}\label{sec_conclusions}
\changed{We investigated the class of algebraic compressed sensing problems where linear or polynomial measurements $\mu : \Var{X} \to \bbk^s$ are made of a model $\Var{X}$, defined implicitly or explicitly by polynomials. Our focus was on theoretical questions of existence, recoverability, identifiability, and continuity. We showed how standard results from classic algebraic geometry and differential geometry can be applied to give almost complete answers to these questions in the case of \emph{generic} algebraic compressed sensing problem instances. The numerical experiments illustrated that all reconstructions in small-scale problem instances without noise can be found using numerical homotopy continuation methods.}

\changed{We hope that the present work can be a starting point for further study of algebraic compressed sensing problems. There are several exciting challenges and directions one may follow. Indeed, we already mentioned (i) proving that a map $\mu: \Var{X}\rightarrow \bbk^s$ given by $s$ polynomials $f_1,\dots,f_s$, where each $f_i$ is generic of degree $a_i$, is everywhere recoverable if and only if $s\geq d$; (ii) deciding if every smooth, complex projective curve may be injected by a morphism of algebraic varieties to a projective plane; and (iii) providing conditions when there exists a generically identifiable coordinate projection into a $(d+1)$-dimensional space.}

\changed{An interesting observation is that one can regard the set of coordinate projections as a $0$-dimensional reducible variety in the space of linear forms. This naturally leads to the following question: what are the algebraic varieties $\Var{Y}$ in the space of linear forms with the property that for every $\Var{X}$, $\dim \Var{X} +1$ many (generic) measurements from $\Var{Y}$ are sufficient to obtain identifiability? We saw that the entire space of linear forms satisfies this condition and that the set of coordinate projections does not. We conjecture that a sufficient condition is that $\Var{Y}$ is irreducible, not linearly degenerate, and of positive dimension.}

\changed{Alternative characterizations or sufficient conditions that certify existence, recoverability, and identifiability of \textit{specific} algebraic compressed sensing problems is another area that merits further study. This likely requires a toolset that extends beyond the one we used here. The case of coordinate projections of low-rank matrices, for example, has been investigated using rigidity theory \cite{SC2010,CMNT2023}, algebraic combinatorics \cite{KTT2015}, and matroid theory \cite{Tsakiris2021,Tsakiris2021b}.}

\changed{Lastly, we focused on theoretical questions of algebraic compressed sensing problems, showing that many generic problem instances are well posed even with very few measurements. We did not investigate how these problems should be solved algorithmically in practice, where one must take into account measurement errors and the propagation of numerical roundoff errors. It would be interesting to investigate whether numerical homotopy continuation, or other numerical methods for solving systems of polynomial equations, can stably solve algebraic compressed sensing problems, not just for generic instances, but for all instances.}

\section*{Acknowledgements}

PB and NV thank Sebastian Kr\"amer for our discussions on the topic of recoverability and identifiability, and the numerical experiments he performed that further support the theory laid out in this paper. NV thanks Manolis Tsakiris for inspiring discussions on deterministic patterns for finite or unique completability of low-rank matrix and tensor models. We also thank him for detailed feedback on an earlier version of this manuscript and suggesting a further simplification of the answer to \cref{question: specific map is recoverable}. Jose Israel Rodriguez is thanked for sharing an additional reference on compressed sensing of moment varieties.

We thank the reviewers for their kind comments and suggestions for improvement, and we thank the editor for handling the process.
 
\bibliographystyle{elsarticle-num}
\bibliography{recovery} 
\end{document}